\documentclass[a4paper,12pt]{article}
\usepackage[latin1]{inputenc}
\usepackage{amsfonts}
\usepackage{amsthm}
\usepackage{amsmath}
\usepackage{amsfonts}
\usepackage{latexsym}
\usepackage{amssymb}
\usepackage{dsfont}
\usepackage{graphicx}
\usepackage[usenames]{color}

\newtheorem{fed}{Definition}[section]
\newtheorem*{fed*}{Definition}
\newtheorem*{feds*}{Definitions}
\newtheorem{teo}[fed]{Theorem}
\newtheorem*{teo*}{Theorem}
\newtheorem{lem}[fed]{Lemma}

\newtheorem{pro}[fed]{Proposition}
\theoremstyle{definition}
\newtheorem{rem}[fed]{Remark}

\newtheorem*{rems*}{Remarks}

\newtheorem{exa}[fed]{Example}

\newtheorem{nota}[fed]{Notation}

\def\coma{\, , \, }

\def\py{\peso{and}}
\newcommand{\peso}[1]{ \quad \text{ #1 } \quad }

\def\n0{n_{ \text{\rm \tiny o}}}

\def\suml{\sum\limits}

\def\bce{\begin{center}}
	\def\ece{\end{center}}

\def\py{\peso{and}}

\def\noi{\noindent}
\def\cF{\mathcal F}
\def\cG{\mathcal G}
\def\QED{\hfill $\square$}
\def\EOE{\hfill $\triangle$}

\def\uno{\mathds{1}}
\def\bm{\left[\begin{array}}
	\def\em{\end{array}\right]}
\def\ben{\begin{enumerate}}
	\def\een{\end{enumerate}}
\def\bit{\begin{itemize}}
	\def\eit{\end{itemize}}
\def\barr{\begin{array}}
	\def\earr{\end{array}}
\def\igdef{\ \stackrel{\mbox{\tiny{def}}}{=}\ }

\def\la{\lambda}

\def\N{\mathbb{N}}
\def\R{\mathbb{R}}
\def\C{\mathbb{C}}
\def\I{\mathbb{I}}

\def\cI{\mathcal{I}}

\def\cH{\mathcal{H}}

\def\cS{{\cal S}}
\def\cT{{\cal T}}
\def\cM{{\cal M}}
\def\cB{{\cal B}}

\def\cV{{\cal V}}
\def\cU{{\cal U}}

\def\ese{\mathcal{S}}

\def\cX{\mathcal{X}}
\def\cY{\mathcal{Y}}
\def\cZ{\mathcal{Z}}

\def\orto{^\perp}
\def\inc{\subseteq}

\def\ua{^\uparrow}
\def\da{^\downarrow}

\def\spr{\text{\rm Spr}}
\def\sprm{\text{\rm Spr}^+}

 \DeclareMathOperator{\tr}{tr}

\DeclareMathOperator{\leqp}{\leqslant}

\def\H{{\cal H}}

\newcommand{\mat}{\mathcal{M}_d(\mathbb{C})}

\newcommand{\matsad}{\mathcal{H}(d)}

\newcommand{\matud}{\mathcal{U}(d)}

\newcommand{\matpos}{\mat^+}

\newcommand{\matrec}[1]{\mathcal{M}_{#1} (\mathbb{C})}

\def\beq{\begin{equation}}
	\def\eeq{\end{equation}}

\def\pausa{\medskip\noi}

\def\Ax2{\,( S_{E(\cF)^\#_\cV})\hat{}_x }

\begin{document}

 \title{Absolute variation of Ritz values,  \\ principal angles and spectral spread}
 \author{Pedro Massey, Demetrio Stojanoff and Sebastián Zárate 
 	\footnote{Partially supported by 
 		CONICET (PIP 0152 CO), FONCyT (PICT-2015-1505) and UNLP (11X829) 
 		e-mail addresses:  massey@mate.unlp.edu.ar , demetrio@mate.unlp.edu.ar , seb4.zarate@gmail.com   }
 	\\
 	{\small Centro de Matem\'atica, FCE-UNLP,  La Plata
 		and IAM-CONICET, Argentina }}
 \date{}
 \maketitle

 	\begin{abstract}
 		\noindent Let $A$ be a $d\times d$ complex self-adjoint matrix, $\mathcal{X},\mathcal{Y}\subset \mathbb{C}^d$ be $k$-dimensional subspaces and let $X$ be a $d\times k$ complex matrix whose columns form an orthonormal basis of $\mathcal{X}$; that is, $X$ is an isometry whose range is the subspace $\cX$.
 		We construct a $d\times k$ complex matrix $Y_r$ whose columns form an orthonormal basis of $\mathcal{Y}$ and obtain sharp upper bounds for the singular values $s(X^*AX-Y_r^*\,A\,Y_r)$ in terms of submajorization relations involving the principal angles between $\mathcal{X}$ and $\mathcal{Y}$ and the spectral spread of $A$. We apply these results to obtain sharp upper bounds for the absolute variation of the Ritz values of $A$ associated with the subspaces $\mathcal{X}$ and $\mathcal{Y}$, that partially confirm conjectures by Knyazev and Argentati.
 	\end{abstract}
 	
 	
 	
 	\tableofcontents

 	\section{Introduction}
 	
 	The problem of computing eigenvalues and invariant subspaces of self-adjoint matrices is ubiquitous in applications of linear algebra and numerical analysis. 
 	For example, given a $d\times d$ complex positive semidefinite matrix $A$, we could be interested in computing a low rank matrix approximation, which is an essential tool to deal with 
 	large and/or sparse self-adjoint matrices (see \cite{Parlett,Saad,StewSun}); in this case,  optimal approximations of $A$ of rank $k$ are constructed in terms of invariant subspaces associated with the $k$ largest eigenvalues of $A$. In general, there is no universal effective solution to these problems ; there are several fundamental algorithmic methods for the computation of {\it approximate} eigenvalues and invariant subspaces, and the best suited such method typically depends on the context. One common fundamental aspect of these algorithms is the assessment of the quality of a subspace (and its associated Ritz values), as a candidate for an approximate invariant subspace (and approximate eigenvalues).
 	Indeed, let $X$ be a
 	$d\times k$ complex matrix whose columns form an orthonormal basis of a $k$-dimensional  subspace $\cX$. Then, $\cX$ is an $A$-invariant subspace if and only if the so-called residual of $A$ at $X$, given by 
 	$R_X=AX-X(X^*AX)$, is the zero matrix. In this case, the Ritz values of $A$ associated with $\cX$ that is, the eigenvalues $\la(X^*AX)=(\la_j(X^*AX))_{j=1}^k$ (counting multiplicities and arranged non-increasingly) of the $k\times k$ submatrix
 	$X^*AX$,  are eigenvalues of $A$. In general, if the residual is {\it small}, then we consider $\cX$ an approximate invariant subspace and its associated Ritz values as approximate eigenvalues of $A$.

 	\pausa
 	The are some other, rather indirect, measures of the quality of a subspace $\cX$ as a possible invariant subspace of the complex $d\times d$ self-adjoint matrix $A$, based of the local behavior of the Ritz values of $A$ associated with 
 	subspaces $\cY$ that are close to $\cX$ (an example of this phenomenon is described at the beginning of Section \ref{sec applicat}). This fact has been one motivation for the study of the so-called 
 	absolute variation of Ritz values (see 
 	\cite{AKPRitz,BosDr,AKProxy,AKMaj,AKFEM,LiLi,Mathias,Ovt,TeLuLi,ZAK,ZK}). In this context, 
 	we are interested in getting upper bounds for (some measure of) the vector
 	\begin{equation} \label{defi abs var rv}
 		|\la(X^*AX)-\la(Y^*AY)|:=(\,|\,\la_j(X^*AX)-\la_j(Y^*AY)\,| \,)_{j=1}^k\in \R^k_{\geq 0}\,,
 	\end{equation} where $X$ and $Y$ are $d\times k$ isometries with ranges $R(X)=\cX$ and $R(Y)=\cY$.
 	By continuity of eigenvalues, the absolute variation of Ritz values is controlled by the distance between the subspaces $\cX$ and $\cY$; as a vector valued measure of distance between the subspaces $\cX$ and $\cY$, we consider the principal angles between $\cX$ and $\cY$, denoted $\Theta(\cX,\cY)=(\theta_j)_{j=1}^k\in [0,\pi/2]^k$, that are the angles whose cosines are the singular values of $X^*Y$ arranged in non-decreasing order and counting multiplicities.

 	\pausa 
 	On the other hand, it turns out that the absolute variation of Ritz values is also controlled by the spread of the eigenvalues of $A$ (independently of the subspaces $\cX$ and $\cY$); notice that in the extreme case in which $A=a\,I_d$ is a multiple of the identity matrix, or equivalently when the spread of the eigenvalues of $A$ is zero, then $X^*AX=a\,I_k=Y^*AY$ so the variation of Ritz values is also zero. 
 	As a vector valued measure of the spread of the eigenvalues of $A$,  Knyazev and Argentati 
 	(see \cite{AKFEM}) have suggested to consider the so-called {\it spectral spread of} $A$, denoted $\sprm(A)$, given by 
 	$$\sprm (A) =  \big(\la_j(A)-\la_{d-j+1}(A)\, \big)_{j=1}^h \in (\R_{\geq 0}^h)\da \,, 
 	$$ where $h = [\frac d2]$ (integer part).
 	In this context, they have conjectured the following (autonomous) {\it a priori} upper bounds 
 	(see \cite[Conjecture 2.1]{AKFEM}): given two $d\times k$ isometries $X$ and $Y$ then
 	\begin{equation} \label{eq conj KA 1}
 		|\la(X^*AX)-\la(Y^*AY)|\prec_w \sin(\Theta(\cX,\cY))\, \sprm(A)\,,
 	\end{equation} where multiplication is performed entry-wise and 
 	$\prec_w$ denotes {\it submajorization}, which is a pre-order relation between real vectors (see Section \ref{sec prelis} for details). 
 	Moreover, if in addition the subspace $\cX$ is $A$-invariant then they have also conjectured that
 	\begin{equation} \label{eq conj KA 2}
 		|\la(X^*AX)-\la(Y^*AY)|\prec_w \sin(\Theta(\cX,\cY))^2\, \sprm(A)\,.
 	\end{equation} 
 	There has been important progress in this direction, and there are several related results and numerical evidence supporting these conjectures (see \cite{AKProxy,AKFEM,MSZ}).
 	
 	\pausa
 	Based on Lidskii's additive inequality, in order to bound the absolute variation of Ritz values we can look for upper bounds of the singular values
 	\begin{equation}\label{defi sing val of var1} 
 		s(X^*AX-Y^*AY)\in\R^k
 		\,,
 	\end{equation}
 	where $s(Z)\in\R^d$ denotes the vector of singular values of a $d\times d$ complez matrix $Z$, counting multiplicities and arranged in non-decreasing order. 
 	As opposed to the variation of Ritz values in Eq. \eqref{defi abs var rv}, 
 	the singular values in Eq. \eqref{defi sing val of var1} actually depend on 
 	the particular choices of partial isometries $X$ and $Y$ with fixed ranges $R(X)=\cX$ and $R(Y)=\cY$.
 	Motivated by geometric insights, given the isometry $X\in\cM_{d,k}(\C)$ with
 	range $R(X)=\cX$ and the subspace $\cY$ as above, 
 	we will choose an explicit (and convenient) $Y=Y_r$
 	and obtain upper bounds for the singular values in Eq. \eqref{defi sing val of var1} in terms of
 	submajorization relations involving the principal angles between $\cX$ and $\cY$ and the spectral spread of $A$ (see Section \ref{sec main results detalles} for a detailed description of our main results regarding this problem). In a sense, our approach corresponds to the study of the stability of the {\it restricted} submatrix extraction process $(A,X,\cY)\mapsto Y_r^*AY_r$ around the fixed data $(X,A)$, where $Y_r=Y_r(X,\,\cY)$ has an explicit dependence that we describe in detail in Section \ref{sec proof of main res}.
 	
 	\pausa Once we have obtained such upper bounds we can apply Lidskii's additive inequality 
 		and obtain the following upper bounds 
 		for the absolute variation of Ritz values: 
 		\begin{equation} \label{eq main result 1 intro abc}
 			|\la(X^*AX)-\la(Y^*AY)|\prec_w \Theta(\cX,\cY)\, \sprm(A)\,,
 		\end{equation} where multiplication is performed entry-wise, and
 		if in addition the subspace $\cX$ is an $A$-invariant subspace: 
 		\begin{equation} \label{eq main result 2 intro abc}
 			|\la(X^*AX)-\la(Y^*AY)|\prec_w \Theta(\cX,\cY)^2\, \sprm(A)\,.
 		\end{equation} 
 		Although our results do not settle the conjectures in their original form, 
 		our upper bounds in 
 		Eqs. \eqref{eq main result 1 intro abc} and \eqref{eq main result 2 intro abc}
 		are comparable with those conjectured in Eqs. \eqref{eq conj KA 1} and \eqref{eq conj KA 2} for the general and the invariant case, specially for perturbations $Y$ of $X$, since 
 		$\lim_{\theta\rightarrow 0^+}\frac{\sin(\theta)}{\theta}=1$. 
 		Moreover, we include a family of examples that show that our vector valued upper bounds are  sharp (see Section \ref{sec applicat}): explicitly, 
 		we show that there exist selections 
 		of $d\times k$ isometries $Y(t)\neq X$ for $t\in (0,1)$ such that 
 		$$
 		\lim_{t\rightarrow 0^+}Y(t)=X\py \lim_{t\rightarrow 0^+}\frac{|\la(X^*AX)-\la(Y(t)^*A\,Y(t))|}{\Theta(\cX,\cY(t))\, \sprm(A)}=(1,\ldots,1)\in\R^k\,.
 		$$ The previous facts suggest that the upper bounds in Eqs. \eqref{eq main result 1 intro abc} and \eqref{eq main result 2 intro abc} could be a useful tool to deal with the case where $Y$ is a perturbation of $X$.

 	\pausa The paper is organized as follows. In section \ref{sec prelis} we introduce the notation and terminology used throughout the paper. In Section \ref{sec 3} we state our main results; indeed, in Section \ref{sec main results detalles} we state our results on the stability of the restricted submatrix extraction process $(A,X,\cY)\mapsto Y_r^*AY_r$ (for $Y_r=Y_r(X,\cY)$ explicitly constructed) around the fixed data $(X,A)$, in terms of the principal angles $\Theta(\cX,\cY)$ and the spectral spread $\sprm(A)$. In Section \ref{sec applicat} we apply the previous results to obtain upper bounds for the absolute variation of Ritz values. In Section \ref{sec proof of main res} we develop the proofs of the results in Section \ref{sec main results detalles}. Our approach is based on bounding the singular values of $X^*AX-Y_r^*AY_r$ in terms of the integral of the singular values of the derivative $\gamma'(t)$ of a smooth curve $\gamma(t)$ joining $X^*AX$ with $Y_r^*AY_r$. An essential part of our argument relies on the construction of a convenient curve $\gamma(t)$. 
 	Then, we apply recent results from \cite{MSZ2} for the spectral spread of self-adjoint matrices to bound the singular values of the derivative $s(\gamma'(t))$.  
 	We have included a short Appendix (Section \ref{appendix}) with some facts related to majorization theory and the spectral spread of self-adjoint matrices, that are used throughout the paper.

 	\section{Preliminaries}\label{sec prelis}

 	Throughout our work we use the following. 
 	
 	\pausa
 	{\bf Notation and terminology}. We let $\mathcal{M}_{d,\,k}(\C)$ be the space of complex $d\times k$ matrices 
 	and write $\mathcal{M}_{d,d}(\C)=\mat$ for the algebra of $d\times d$ complex matrices. 
 	We denote by $\H(d)\subset \mat$ the real subspace of Hermitian (self-adjoint) matrices, by $i\cdot\cH(d)$ the space of skew-Hermitian matrices and by $\matpos$, the cone of
 	positive semi-definite matrices. Also, 
 	$\mathcal{U}(d) \subset \mat$ denotes the group of unitary matrices.
 	Given $1\leq k\leq d$ we denote $\cI(k,d)$ the set of isometries $X\in \cM_{d,\,k}(\C)$ i.e. such that $X^*X=I_k$; notice that $\cI(k,d)$ can be identified in a natural way with the complex Stiefel manifold. Moreover, if $\cX\subset\C^d$ is a $k$-dimensional subspace of $\C^d$ we let 
 	$\cI_\cX(k,d)$ denote those isometries $X\in \cM_{d,\,k}(\C)$ with range $R(X)=\cX$.

 	\pausa
 	For $d\in\N$, let $\I_d=\{1,\ldots,d\}$. 
 	Given a vector $x\in\C^d$ we denote by $D_x=\text{diag}(x_1,\ldots,x_d)$ the diagonal matrix in $\mat$ whose main diagonal is $x$.
 	Given $x=(x_i)_{i\in\I_d}\in\R^d$ we denote by $x\da=(x_i\da)_{i\in\I_d}$ the vector obtained by 
 	rearranging the entries of $x$ in non-increasing order. We also use the notation
 	$(\R^d)\da=\{x\in\R^d\ :\ x=x\da \}$ and $(\R_{\geq 0}^d)\da=
 	\{x\in\R_{\geq 0}^d\ :\ x=x\da \}$. For $r\in\N$, we let $\uno_r=(1,\ldots,1)\in\R^r$.
 	
 	\pausa
 	Given a matrix $A\in\matsad$ we denote by $\la(A)=(\la_i(A))_{i\in\I_d}\in (\R^d)\da$ 
 	the eigenvalues of $A$ counting multiplicities and arranged in 
 	non-increasing order.   
 	For $B\in\mat$, $s(B)=\la(|B|)\in (\R_{\geq 0}^d)\da$ denotes the singular values of $B$, i.e. the eigenvalues of $|B|=(B^*B)^{1/2}\in\matpos$. 
 	Arithmetic operations with vectors are performed entry-wise i.e., in case $x=(x_i)_{i\in\I_k},\,y=(y_i)_{i\in\I_k}\in \C^k $ 
 	then $x+y=(x_i+y_i)_i$, $x\, y=(x_i\,y_i)_i$ and (assuming that $y_i\neq 0$, for $i\in\I_k$) $x/y=(x_i/y_i)_i$, 
 	where these vectors all lie in $\C^k$. Moreover, if we assume further which $x,\,y\in\R^k$ then we write $x\leqp y$ whenever 
 	$x_i\leq y_i$, for $i\in\I_k$.
 	Finally, given a function $f:I\rightarrow \R$ defined on $I\subseteq\R$ and $x=(x_i)_{i\in\I_k}\in I^k$ then we set $f(x)=(f(x_i))_{i\in\I_k}\in \R^k$.\EOE
 	
 	\pausa Next we recall the notion of majorization between vectors, which  will play a central role throughout our work.
 	\begin{fed}\rm 
 		Let $x,\, y\in\R^k$. We say that $x$ is
 		{\it submajorized} by $y$, and write $x\prec_w y$,  if
 		$$
 		\suml_{i=1}^r x^\downarrow _i\leq \suml_{i=1}^r y^\downarrow _i \peso{for every} r\in\I_k\,. 
 		$$ If $x\prec_w y$ and $\tr x \igdef \suml_{i=1}^kx_i=
 		\tr y$,  then we say that $x$ is
 		{\it majorized} by $y$, and write $x\prec y$. \EOE
 	\end{fed}
 	\pausa We point out that (sub)majorization is a preorder relation in $\R^k$ that plays a central role in matrix analysis (see Section \ref{subsec append mayo}).
 	
 	\begin{rem}\label{rem acuerdos}\rm
 		Let $x\in\R_{\geq 0}^k$ and $y\in \R_{\geq 0}^h$ be two vector with {\it non-negative} entries (of different sizes). We extend the notion of submajorization, sum  and the product between $x$ and $y$ in the following sense: 
 		Let $0_n$ denotes the zero vector of $\R^n$ and $\ell:=\max\{h\coma k\}$.  
 		\ben
 		\item We say that $x$ is submajorized by $y$ 
 		if 
 		\begin{equation}\label{size}
 			x\prec_w y \peso{if} \begin{cases} (x\coma 0_{h-k}) \prec_w \quad \ y  & \peso{for} k<h \\ 
 				\quad \quad x  \quad \  \prec_w (y\coma 0_{k-h})   & \peso{for} h<k \end{cases} 
 			\quad , 
 		\end{equation} 
 		\item  Similarly we define $x+y $ and $x \, y \in \R_{\geq 0}^{\ell}\,$, 
 		adding zeros to the right to get two vectors with equal size. \EOE
 		\een
 		
 	\end{rem}
 	\section{On principal submatrices, angles and spectral spread}\label{sec 3}
 	
 	In this section we state our main results. Indeed, in Section \ref{sec main results detalles} we
 	state our results on the stability of the restricted submatrix extraction process $(A,X,\cY)\mapsto Y_r^*AY_r$, where $Y_r=Y_r(X,\cY)$ is defined as $Y_r=U\,X$ for a direct rotation $U$ from $\cX$ to $\cY$. In this case, we bound the vector of singular values $s(X^*AX-Y_r^*AY_r)$ (counting multiplicities and arranged non-increasingly) in terms of the principal angles $\Theta(\cX,\cY)$ and the spectral spread $\sprm(A)$. In Section \ref{sec applicat} below we apply the previous results to obtain upper bounds for the absolute variation of Ritz values.
 	
 	\subsection{On the variation of principal submatrices}\label{sec main results detalles}
 	
 	\begin{rem}\label{pr ang}\rm
 		We begin by recalling some of the notions involved in the statements of our main results.
 		Given $k$-dimensional subspaces $\cX,\,\cY\subset \C^d$, we denote by $\cI_\cX(k,d)$  the set of isometries $X\in \cM_{d,\,k}(\C)$ with range $R(X)=\cX$ (similarly $\cI_\cY(k,d)$\,), 
 		and we consider the {\it principal angles} between
 		$\cX$ and $\cY$, denoted $\Theta(\cX,\cY)=(\theta_j)_{j\in\I_k}\in[0,\pi/2]^k$, given by 
 		\begin{equation}\label{theta}
 			\cos(\Theta)=(\cos(\theta_j))_{j\in\I_k}=(s_{k-j+1}(X^*Y))_{j\in\I_k} \ , 
 		\end{equation}
 		where $X\in\cI_\cX(k,d)$ and  $Y\in\cI_\cY(k,d)$. By construction, $\Theta(\cX,\cY)=\Theta(\cX,\cY)\da$. On the other hand we also consider {\it direct rotations} between subspaces, introduced by Davis and Kahan in \cite{DavKah} (also see \cite{GoLoa}); for the purposes of this work, it is convenient to describe such a concept in a direct and suitable way. In order to do this, we introduce
 			a series of subspaces naturally associated with $\cX$ and $\cY$, that decompose $\C^d$ into mutually orthogonal components; then we describe the direct rotations from $\cX$ onto $\cY$ in terms of block representations with respect to the previous orthogonal decomposition of $\C^d$ (see Eq. \eqref{eq defi rotacion directa prelis}).
 		Indeed, 
 		we consider the generic part in the decomposition of $\C^d$ in terms of the two subspaces $\cX$ and $\cY$ (see Section \ref{subsec main1}), defined  as the subspace 
 		\begin{equation}\label{eq defi part gen}
 			\cG \igdef \big[\,(\cX\cap\cY) \oplus (\cX\cap \cY^\perp)\oplus 
 			(\cX^\perp\cap \cY)\oplus (\cX^\perp\cap \cY^\perp)\,\big]\orto  \ . 
 		\end{equation} 
 		Denote by $p=\dim \cX\cap \cY^\perp=\dim \cX^\perp\cap \cY$, \ 
 		$r=\dim \cX\cap \mathcal{G}= \dim \cX^\perp\cap \mathcal{G} = \dim \cY\cap \cG$, 
 		$$
 		\ese_1 = (\cX\cap \cY^\perp)\oplus (\cX\cap\,\cG) \ \coma \ 
 		\ese_2 = (\cX^\perp\cap \,\cY)\oplus (\cX^\perp\cap\, \cG) \py 
 		\cS_3=  (\cX^\perp\cap\, \cY)\oplus(\cY\cap \,\cG)  \ .
 		$$
 		Notice that $\ese_1 \inc \cX$, $  \ese_2 \inc \cX\orto$ and $\ese_3 \inc \cY$.   
 		Consider the angles  $\Theta'=\Theta(\ese_1\coma \ese_3)\in[0,\pi/2]^{p+r}$, 
 		and the diagonal matrices 
 		$$
 		C=\text{diag}(\,\cos(\Theta')\,)\py S=\text{diag}(\,\sin(\Theta')\,)\in\cM_{p+r}(\C)^+ \ .
 		$$
 		Finally, we say that a unitary $U\in\cU(d)$ is a direct rotation from $\cX$ onto $\cY$ if there exists $\cB$ an orthonormal basis (ONB)  of $\C^d$ obtained by juxtaposition of ONB's for 
 		$\cX\cap\cY \coma \ese_1\coma \ese_2$ and $\cX^\perp\cap \cY^\perp$
 		such that the block matrix representation of $U$ with respect to $\cB$ and these  subspaces is 
 		\begin{equation}\label{eq defi rotacion directa prelis}\begin{pmatrix}
 				I&0 &0 & 0\\
 				0&C&-S & 0\\
 				0&S&C & 0\\
 				0&0&0 &I
 			\end{pmatrix}
 			\barr{l} \cX\cap\cY \\ \ese_1\\ \ese_2\\ \cX^\perp\cap \cY^\perp \earr
 		\end{equation}
 		Moreover, we remark that in this case, if $s=\dim\cX\cap \cY$, then 
 		\begin{equation} \label{los thetas}
 			\Theta'=\left(\ \frac{\pi}{2}\,\uno_p\coma \theta_1\coma \ldots \coma \theta_{r}\, \right) \py
 			\Theta (\cX\coma \cY) =(\Theta',0_{s}) \ .
 		\end{equation}
 		Notice that since $\cX\coma \cY\subset \C^d$ and $\dim \cX =\dim \cY$, a direct rotation $U$ between $\cX$ and $\cY$ always exists. This relies on the fact that $p=\dim \cX\cap \cY^\perp=\dim \cX^\perp\cap \cY$ (see \cite{DavKah} for more details).
 		On the other hand, recall that given a self-adjoint matrix
 		$A\in \cH(d)$ then the spectral spread of $A$, denoted $\sprm (A)$, is given by
 		$$\sprm (A) =  \big(\spr_j(A)\, \big)_{j\in \I_h}
 		=\big(\la_j(A)-\la\ua_j(A)\, \big)_{j\in \I_h} \in (\R_{\geq 0}^h)\da \,, 
 		$$ where $h = [\frac d2]$ (integer part). See section \ref{spread sec} 
 		for more information about this notion. 
 		\EOE
 	\end{rem}

 	\pausa
 	In the following formulae we operate among vectors with non-negative entries 
 	of different sizes, using the notation given in Remark \ref{rem acuerdos}. 
 	The proofs of Theorems \ref{teo principal1} and \ref{teo principal2} 
 	below are developed in Section \ref{sec proof of main res}.

 	\begin{teo} \label{teo principal1}
 		Let $A,\, B\in\cH(d)$ and let $\cX,\,\cY\subset \C^d$ be $k$-dimensional subspaces.
 		Let $U=U(\cX,\cY)$ be a direct rotation of $\cX$ onto $\cY$ and $\Theta=\Theta(\cX,\cY)\da  \in [0,\pi/2]^k$ 
 		the principal angles between $\cX$ and $\cY$ defined in \eqref{theta}. Given 
 		$X\in\cI_\cX(k,d)$, if we let 
 		$$Y_r=Y_r(X,U)\igdef UX\in \cI_\cY(k,d)$$ then we have that 
 		\begin{equation}\label{eq teo main1}
 			s(X^*\,A\,X-Y_r^*\,B\,Y_r )\prec_w s(A-B)+ 	
 			\Theta(\cX,\cY)\da\, \left( \frac{ \,\sprm \,(A)  + \sprm\, (B)}{2}\right)\,,
 		\end{equation} 
 		where submajorization, sums and products are as in Remark \ref{rem acuerdos}. 
 		\QED
 	\end{teo}
 	
 	\pausa
 	Notice that we have considered the more general situation of two self-adjoint matrices $A,\,B\in\cH(d)$. In the next section, we apply this result in the particular case $A=B$ to obtain upper bounds for the absolute variation of the Ritz values in the self-adjoint case.
 	
 	\pausa In the following result we deal with the so-called invariant case, and obtain a stronger upper bound for small perturbations $\cY$ of $\cX$.
 	
 	\begin{teo} \label{teo principal2}
 		With the same hypothesis and notation of Theorem \ref{teo principal1}, assume further that the $k$-dimensional subspace $\cX$ is $A$-invariant. 
 		Then we have that
 		\begin{equation}\label{eq teo main2}
 			s(X^*\,A\,X-Y_r^*\,A\,Y_r )\prec_w \Theta^2(\cX,\cY)\da \, \sprm \, (A)\,,
 		\end{equation} 
 		where submajorization and products are as in Remark \ref{rem acuerdos}. 
 		\QED
 	\end{teo}
 	
 	\pausa
 	In the next section, we apply Theorems \ref{teo principal1} and \ref{teo principal2} and obtain upper bounds for the absolute variation of the Ritz values in the self-adjoint case.
 	
 	\subsection{An application: on the absolute variation of Ritz values}\label{sec applicat}
 	
 	As already mentioned in the Introduction, one motivation for the study of absolute variation of Ritz values of $A\in\cH(d)$ around a subspace $\cX$ in $\C^d$, is that such variation provides an indirect measure of the quality of $\cX$ as a possible invariant subspace of $A$. 
 	As an example of this phenomenon, we recall the following inequality, recently obtained in \cite{MSZ}: let $P_{\cX+\cY}$ denote the orthogonal projection onto the subspace $\cX+\cY$ of $\C^d$, let 
 		$R_X=AX-X(X^*AX)$ and $R_Y=AY-Y(Y^*AY)$ denote the residuals of $A$ at $X$ and $Y$,
 	let $s(P_{\cX+\cY}\  R_X),\, s(P_{\cX+\cY}\ R_Y)\in \R^k$ denote the vectors of singular values of $P_{\cX+\cY}\  R_X$ and $P_{\cX+\cY}\ R_Y$ (counting multiplicities and arranged non-increasingly); if $\cX$ and $\cY$ are in a acute relative position (i.e. so that $\theta_j<\pi/2$, for $j\in\I_k$) then
 	\begin{equation} \label{eq de MSZ1}
 		|\la(X^*AX)-\la(Y^*AY)| 
 		\prec_w [s(P_{\cX+\cY}\  R_X) + s(P_{\cX+\cY}\ R_Y)] \,(\tan(\theta_j))_{j=1}^k\,.
 	\end{equation}
 	In case $\cX$ and $\cY$ are close to each other, then the residuals of $A$ at $X$ and $Y$ are comparable; moreover, since the tangents $\tan(\theta_j)\approx \theta_j$ are comparable, for $j\in\I_k$, then conclude that the inequality in Eq. \eqref{eq de MSZ1} provides an upper bound for the first order absolute variation of the Ritz values of $A$ around $\cX$, in terms of the residual $R_X$ (where we measured the distance between subspaces in terms of the principal angles). Hence, if 
 	$R_X$ is small, then this first order variation (that can be tested numerically) is also small.

 	\pausa We are interested in obtaining {\it autonomous}
 	upper bounds for the absolute variation of Ritz values of $A\in\cH(d)$ associated with the $k$-dimensional subspaces $\cX$ and $\cY$, in terms of the principal angles between $\cX$ and $\cY$, and the spectral spread of the matrix $A$. In this context, Knyazev and Argentati have conjectured (see \cite[Conjecture 2.1]{AKFEM}) the upper bounds in Eq. \eqref{eq conj KA 1} (for the general subspaces)
 	and \eqref{eq conj KA 2} (for an $A$-invariant subspace $\cX$).
 	The following results partially confirm these conjectures.

 	\begin{teo}\label{teo applica1}
 		Let $A\in\cH(d)$ and let $X,\, Y\in \cI(k,d)$, with ranges $\cX=R(X)$ and $\cY=R(Y)$ such that $\dim \cX=\dim\cY=k$. If $\Theta=\Theta(\cX,\cY)$ then:
 		\ben
 		\item We have that 
 		\begin{equation}\label{eq teo aplic1}
 			|\la(X^*AX)-\la(Y^*AY)|\prec_w \Theta\, \sprm (A)\,.
 		\end{equation}
 		\item If we further assume that $\cX$ is $A$-invariant, we have that 
 		\begin{equation}\label{eq teo aplic3}
 			|\la(X^*AX)-\la(Y^*AY)|\prec_w  \Theta^2\, \sprm(A)\,.
 		\end{equation}
 		\een
 	\end{teo}
 	
 	\begin{proof}
 		Let $U\in\cU(d)$ be a direct rotation from $\cX$ onto $\cY$ and let $Y_r=U\,X$ 
 		be as in Theorem \ref{teo principal1}. Since $R(Y_r)=R(Y)$ 
 		then $V = Y_r^*Y \in \cU(k)$, and we get that 
 		$$
 		\la(Y^*A\,Y)=
 		\la(V \, Y^*A\,Y \, V^*)= \la(Y_r^*\, P_\cY \, A\,P_\cY \,  Y_r)
 		= \la(Y_r^*\, A\, Y_r)\in \R^k\,.$$
 		By Lidskii's inequality (Theorem \ref{teo ah}) and Theorem \ref{teo principal1} we have that 
 		\begin{eqnarray*}
 			|\la(X^*AX)-\la(Y^*AY)|&=&|\la(X^*AX)-\la(Y_r^*\,A\,Y_r)|\\  &\prec_w&|\la(X^*AX-Y_r^*A\,Y_r)|\da\\
 			& = &s(X^*AX-Y_r^*A\,Y_r) 
 			\prec_w \Theta\, \sprm (A)\,.
 		\end{eqnarray*}
 		Item 2 follows from a similar argument, using Theorem \ref{teo principal2}.
 	\end{proof}
 	
 	\pausa
 	\begin{rem}\rm
 		The only difference between the Conjectures \eqref{eq conj KA 1} and \eqref{eq conj KA 2}
 		and our main result Theorem \ref{teo applica1} relies in the slightly bigger
 		upper bounds $\Theta$ (resp. $\Theta ^2$) instead of $\sin \Theta$ 
 		(resp. $\sin ^2 \Theta$). These numbers are asymptotically close when the angles 
 		$\Theta \rightarrow 0^+$, and they can be globally compared with a constant $\pi/2$. 
 		
 		\pausa
 		Nevertheless, Conjectures \eqref{eq conj KA 1} and \eqref{eq conj KA 2} are 
 		supported by large computational experimentation. 
 		In order to explore the content of our main results we consider the following examples, 
 		where we show that both Conjectures \eqref{eq conj KA 1} and \eqref{eq conj KA 2} and 
 		Theorem \ref{teo applica1} are sharp inequalities.
 		\EOE
 	\end{rem}
 	
 	\begin{exa}\rm
 		Consider $a>b>0$ and let 
 		$$
 		A=\begin{pmatrix}
 			0 & 0 & a & 0 \\
 			0 & 0 & 0 & b \\
 			a & 0 & 0 & 0 \\
 			0 & b & 0 & 0 
 		\end{pmatrix}\in \cH(4)\,.
 		$$If $\{e_1,\,e_2,\,e_3,\,e_4\}$ denotes the canonical basis of $\C^4$, we let 
 		$\cX=\text{Span}\{e_1,\,e_2\}$. 
 		On the other hand, given $\theta\in[0,\pi/2]$ let 
 		\begin{equation}\label{eq defi fs}
 			f_1(\theta)=\cos(\theta)\,e_1+\sin(\theta)\,e_3\py f_2(\theta)=\cos(\theta)\,e_2+\sin(\theta)\,e_4\,.
 		\end{equation}
 		Then, we set $\cY(\theta)=\text{Span}\{f_1(\theta),\,f_2(\theta)\}$, for $\theta\in[0,\pi/2]$.
 		It is straightforward to show that $\Theta(\cX\coma \cY(\theta))=(\theta,\,\theta)$. Then, it follows that the isometries $X,\,Y_r(\theta)$ as in Theorem \ref{teo principal1} (associated with the subspaces $\cX$ and $\cY(\theta)$) are given by 
 		\begin{equation}\label{eq defi las isometrias}
 			X=\begin{pmatrix}
 				1 & 0 \\
 				0&1\\
 				0&0\\
 				0&0
 			\end{pmatrix}
 			\py 
 			Y_r(\theta)=\begin{pmatrix}
 				\cos(\theta) & 0 \\
 				0&\cos(\theta)\\
 				\sin(\theta) &0\\
 				0&\sin(\theta)
 			\end{pmatrix}\,.
 		\end{equation} Direct computations with the matrices described above show that 
 		$$
 		X^*AX=0 \py Y_r(\theta)^*AY_r(\theta)=
 		\begin{pmatrix}
 			a\,\sin(2\theta) & 0 \\
 			0 & b\,\sin(2\theta)  
 		\end{pmatrix}
 		$$
 		Hence, in this case we get that for $\theta\in[0,\pi/2]$, 
 		$$
 		|\la(X^*AX)-\la(Y_r(\theta)^*AY_r(\theta))|=|\la(X^*AX-Y_r(\theta)^*AY_r(\theta))|=\sin(2\theta)\,(a,b)\,
 		$$ so in particular, Lidskii's inequality holds with equality in this case.
 		On the other hand, it turns out that 
 		$\la(A)=(a,b,-b,-a)\in(\R^4)\da$ which shows that $\sprm(A)=2\,(a,b)$. Therefore, the inequalities in Theorems \ref{teo principal1} (with $A=B$) and \ref{teo applica1} (item 1.) become 
 		$$
 		\sin(2\theta)\,(a,b)\prec_w 2\,(\theta\coma \theta)\,(a,b)=2\,\theta\,(a,b)\, , 
 		$$ where the submajorization relation above is equivalent to the inequalities
 		$$
 		\sin(2\theta)\,a\leq 2\,\theta \,a\py  \sin(2\theta)\,(a+b)\leq 2\,\theta \,(a+b)
 		$$
 		which are sharp; this last claim can be seen by considering $\theta\rightarrow 0^+$.
 		Notice that we further get that 
 			$$
 			\lim_{\theta\rightarrow 0^+}\frac{|\la(X^*AX)-\la(Y_r(\theta)^*AY_r(\theta))|}{\Theta(\cX,\,\cY(\theta))\,\sprm(A)}=\uno_2\,
 			$$ where we have considered the entry-wise quotient of the vectors. This last fact shows that our upper bound for the absolute variation of Ritz values is sharp.
 		\EOE
 	\end{exa}
 	
 	\begin{exa} \rm
 		Consider $a>b>0$ and let 
 		$$
 		A=\begin{pmatrix}
 			a & 0 & 0 & 0 \\
 			0 & b & 0 & 0 \\
 			0 & 0 & 0 & 0 \\
 			0 & 0 & 0 & 0 
 		\end{pmatrix}\in \cH(4)\,.
 		$$ We consider the canonical basis $\{e_1,\,e_2,\,e_3,\,e_4\}$ of $\C^4$, and we let 
 		$\cX=\text{Span}\{e_1,\,e_2\}$. Notice that in this case $\cX$ is an $A$-invariant subspace. On the other hand, given $\theta\in[0,\pi/2]$ let 
 		$\cY(\theta)=\text{Span}\{f_1(\theta),\,f_2(\theta)\}$, where $f_j(\theta)$ are as in Eq. \eqref{eq defi fs}. Hence, as in the previous example, we have that $\Theta(\cX\coma \cY(\theta))=(\theta,\,\theta)$. In this case, the isometries $X,\,Y_r(\theta)$ as in Theorem \ref{teo principal2} are given by Eq. \eqref{eq defi las isometrias}.
 		Direct computations with the matrices described above show that 
 		$$
 		X^*AX=\begin{pmatrix}
 			a & 0 \\
 			0 & b
 		\end{pmatrix}
 		\py Y_r(\theta)^*AY_r(\theta)=
 		\begin{pmatrix}
 			a\,\cos^2(\theta) & 0 \\
 			0 & b\,\cos^2(\theta)  
 		\end{pmatrix}
 		$$
 		Hence, in this case we get that for $\theta\in[0,\pi/2]$, 
 		$$
 		|\la(X^*AX)-\la(Y_r(\theta)^*AY_r(\theta))|=|\la(X^*AX-Y_r(\theta)^*AY_r(\theta))|=\sin^2(\theta)\,(a,b)\,
 		$$ so in particular, Lidskii's inequality holds with equality in this case.
 		On the other hand, it turns out that 
 		$\la(A)=(a,b,0,0)\in(\R^4)\da$ which shows that $\sprm(A)=(a,b)$. Therefore, the inequalities in Theorems \ref{teo principal2}  and \ref{teo applica1} (item 2.) become 
 		$$
 		\sin^2(\theta)\,(a,b)\prec_w (\theta\coma \theta)^2\,(a,b)=\theta^2\,(a,b)\, , 
 		$$ where the submajorization relation above is equivalent to the inequalities
 		$$
 		\sin^2(\theta)\,a\leq \theta^2 \,a\py  \sin^2(\theta)\,(a+b)\leq \theta^2 \,(a+b)
 		$$
 		which are sharp; this last claim can be seen by considering $\theta\rightarrow 0^+$. 
 		Notice that we further get that 
 			$$
 			\lim_{\theta\rightarrow 0^+}\frac{|\la(X^*AX)-\la(Y_r(\theta)^*AY_r(\theta))|}{\Theta(\cX,\cY(\theta))^2\,\sprm(A)}=\uno_2\,,
 			$$where we have considered the entry-wise quotient of the vectors. This last fact shows that our upper bound for the absolute variation of Ritz values is sharp.
 			 		\EOE
 	\end{exa}

 	\pausa As a final comment, notice that it's natural to wonder whether the estimates
 	from Theorems \ref{teo principal1} and \ref{teo principal2} can be used to obtain
 	estimates for the distance between the compressions (that we can think of as lower rank approximations of $A$) given by $P_\cX\,A\,P_\cX$ and 
 	$P_\cY\,A\,P_\cY$. It turns out that this is not the case. Indeed, in the trivial case in which
 	$A=I$ then we have that $P_\cX\,I\,P_\cX-P_\cY\,I\,P_\cY=P_\cX-P_\cY$; but we have that 
 	$\sprm(I)=0$ is the zero vector, so there is no hope in obtaining an upper bound for  $P_\cX\,A\,P_\cX-P_\cY\,A\,P_\cY$ in terms $\sprm(A)$ in general. Also, in general there is no dependence of $P_\cX\,A\,P_\cX-P_\cY\,A\,P_\cY$ in terms of $\Theta(\cX,\cY)^2$ when $\cX$ is $A$-invariant (again, take $A=I$ to see this). 
 	
 	\section{Proof of the main results}\label{sec proof of main res}
 	
 	In this section we present complete proofs of our main results. Our approach is based on some
 	geometric arguments and inequalities for the spectral spread recently obtained 
 	in \cite{MSZ2}.

 	\subsection{Proof of Theorem \ref{teo principal1}}\label{subsec main1}
 	
 	Throughout this section we adopt the notation and terminology in Theorem \ref{teo principal1}. Hence,
 	we consider:
 	\ben
 	\item $A,\, B\in\cH(d)$;
 	\item  $\cX,\,\cY\subset \C^d$, $k$-dimensional subspaces and $X\in\cI_\cX(k,d)$;
 	\item a direct rotation $U=U(\cX,\cY)\in \matud$, of $\cX$ onto $\cY$.
 	\item $Y_r=Y_r(X,\cY)\igdef UX\in \cI_\cY(k,d)$.
 	\een
 	\pausa
 	Our approach to prove Theorem \ref{teo principal1} is as follows: we will 
 	consider smooth curves
 	\begin{equation}\label{eq defi curvas}
 		L(\cdot):[0,1]\rightarrow\cH(d) \py Y_r(\cdot):[0,1]\rightarrow \cI(k,d)
 	\end{equation}
 	such that $L(0)=A$, $L(1)=B$, $Y_r(0)=X$ and $Y_r(1)=Y_r$. Then, we consider the smooth curve
 	\begin{equation}\label{eq defi gamma1}
 		\gamma:[0,1]\rightarrow \cH(k)\peso{given by} \gamma(t)=Y_r(t)^*\, L(t)\, Y_r(t) \peso{for} t\in[0,1]\,.
 	\end{equation}
 	Once we have constructed $\gamma(\cdot)$ we will apply the following result.
 	Recall that given $Z\in\mat$, $s(Z)\in\R^d$ denotes the vector of singular values, counting multiplicities and arranged in non-increasing order.
 	\begin{pro}\label{prop estimando distancias1} 
 		Let $\gamma:[0,1]\rightarrow \matrec{m\coma n}$ be a smooth curve such that $\gamma(0)=C$ and $\gamma(1)=D$. Then
 		$$
 		s(D-C)\prec_w \int_0^1 s(\gamma'(t))\ dt\,.
 		$$
 	\end{pro}
 	\begin{proof} First notice that by the fundamental theorem of calculus we have that
 			\begin{equation}\label{eq para prop 411}
 				D-C=\gamma(1)-\gamma(0)=\int_0^1\gamma'(t)\ dt\,.
 			\end{equation}On the other hand, for $n\in\N$ consider
 			the regular partition $\{t_0=0<t_1<\ldots<t_n=1\}$ of $[0,1]$ so that $t_j=\frac{j}{n}$, for $j\in\{0\}\cup\I_n$ and $\Delta_j=t_{j}-t_{j-1}=\frac 1 n=\Delta_n$, for $j\in\I_n$. Then, 
 			\begin{equation}\label{eq para prop 412}
 				\int_0^1\gamma'(t)\ dt=\lim_{n\rightarrow\infty} \sum_{j=1}^n \gamma'(t_j)\ \Delta_n
 				\ \ , \ \ \int_0^1 s(\gamma'(t))\ dt=\lim_{n\rightarrow\infty} \sum_{j=1}^n s(\gamma'(t_j))\ \Delta_n
 			\end{equation}
 			where we have used that the curves $\gamma'(t)\in\matrec{m\coma n} $ and $s(\gamma'(t))\in\R^d$, for $t\in[0,1]$, are continuous (by hypothesis and by the continuity of singular values). By Weyl's additive inequality (see Theorem \ref{teo ag}) we have that 
 			\begin{equation}\label{eq para prop 413}
 				s(\sum_{j=1}^n \gamma'(t_j)\ \Delta_n)\prec_w
 				\sum_{j=1}^n s(\gamma'(t_j))\ \Delta_n
 			\end{equation}
 			The result now follows from Eqs. \eqref{eq para prop 411}, \eqref{eq para prop 412}, \eqref{eq para prop 413}, the continuity of singular values and the following fact: if 
 			$(x_n)_{n\in\N}$ and $(y_n)_{n\in\N}\in\R^d$ are sequences that converge to $x$ and $y\in\R^d$ respectively, and such that $x_n\prec_w y_n$ for $n\in\N$, then $x\prec_w y$.
 	\end{proof}

 	\pausa
 	Although the arguments considered in the previous paragraphs are valid for 
 	any choice of smooth curves $L(\cdot)$ and $Y_r(\cdot)$ as in Eq. \eqref{eq defi curvas}, we are interested in
 	choices that lead to better upper bounds. Thus, we are interested in those curves 
 	$L(\cdot)$ and $Y_r(\cdot)$ for which the associated curve $\gamma(\cdot)$
 	in Eq. \eqref{eq defi gamma1} is {\it minimal} in a certain sense. We point out that we will not study the corresponding minimality problem, but rather we will choose 
 	$L(\cdot)$ and $Y_r(\cdot)$ that have separately minimal properties, and use these choices to build 
 	$\gamma(\cdot)$.
 	
 	\pausa
 	For $L(\cdot)$ there is a natural choice, namely the line segment joining $A$ and $B$, i.e.
 	\begin{equation}\label{eq defi Adt}
 		L(t)= (1-t)\, A + t\, B\peso{for} t\in[0,1]\,.
 	\end{equation}
 	Next we construct $Y_r(t)$, based on the notion of direct rotation as developed in \cite{DavKah};
 	hence, we consider the following notions related to the direct rotation of
 	$\cX$ onto $\cY$ as in Remark \ref{pr ang}. 
 	Indeed, given these subspaces we have the orthogonal decomposition (see \cite{Hal})
 	$$
 	\C^d=(\cX\cap \cY) \oplus( \cX\cap \cY^\perp) \oplus \mathcal{G}\oplus(\cX^\perp \cap \cY)\oplus 
 	(\cX^\perp \cap \cY^\perp)\,.
 	$$
 	Here $\cG\subseteq\C^d$ stands for the generic part of the pair of subspaces $\cX$ and $\cY$ (see Eq. \eqref{eq defi part gen}). We point out that some of these subspaces can be null. 
 	It turns out that for our purposes, we can assume further (see the proof of Theorem 
 	\ref{teo principal1} below) that $$\cX^\perp\cap \cY^\perp=\{0\}\,.$$
 	Recall that we can decompose the generic part $\cG=\cX\cap \mathcal{G}\oplus \cX^{\perp} \cap \mathcal{G}$
 	into two subspaces such that $\dim \cX\cap \mathcal{G}=\dim \cX^\perp\cap \mathcal{G}=r$. Therefore, in our case we get the decomposition
 	\begin{equation}\label{eqdecompC0} \C^d=
 		(\cX\cap \cY )\oplus\cS_1\oplus \cS_2\,,  \end{equation} 
 	where $$\cS_1=(\cX\cap \cY^\perp)\oplus (\cX\cap \mathcal{G}) \py \cS_2=(\cX^\perp\cap \cY)\oplus (\cX^{\perp} \cap \mathcal{G})\,.$$ Since $\dim\cX=\dim\cY$ we conclude that $\dim \cX\cap \cY^\perp=\dim \cX^\perp\cap \cY=p$ so, in particular, $\dim\cS_1=\dim \cS_2=r+p$. 
 	
 	\pausa Let $U$ be a direct rotation of $\cX$ onto $\cY$. 
 	In this case, there exist orthonormal bases of $\cS_1$ and $\cS_2$
 	and diagonal positive semidefinite matrices 
 	$$C=\text{diag}(\,\cos(\Theta')\,)\py S=\text{diag}(\,\sin(\Theta')\,)\in\cM_{p+r}(\C)^+ $$ where $\Theta'=(\frac{\pi}{2}\,\uno_p\coma \theta_1,\ldots,\theta_{r})$ denote the principal angles between 
 	$\cS_1$ and $(\cY\cap \cG)\oplus( \cX^\perp\cap \cY)$, such that  
 	\begin{equation}\label{eq defi dir rot u1}
 		U=\begin{pmatrix}
 			I&0 &0\\
 			0&C&-S\\
 			0&S&C\end{pmatrix}\,,\end{equation}  where the block matrix representation is with respect to the ONB obtained by juxtaposition of the ONB's of
 	$\cX\cap\cY$, $\cS_1$ and $\cS_2$. Here $(\theta_1,\ldots,\theta_{r})\in (0,\pi/2)^r$ denote the principal angles between $\cX\cap \mathcal{G}$ and $\cY \cap \mathcal{G}$.
 	We now define
 	\begin{equation}\label{eq defi udt}
 		U(t) 
 		=\begin{pmatrix}
 			I&0 &0\\
 			0&C(t)&-S(t)\\
 			0&S(t)&C(t)\end{pmatrix}\in\matud\,,\end{equation}
 	where $S(\cdot),\,C(\cdot): \ [0\, ,\,1] \rightarrow \H(p+r)$ are given by 
 	\begin{equation}\label{eq defi sycdt}
 		S(t)=\text{diag}(\,\sin(t\,\Theta')\,)\py C(t)=\text{diag}(\,\cos(t\,\Theta')\,)\in\cM_{p+r}(\C)^+\,.
 	\end{equation}
 	Then $U(0)=I$ and $U(1)=U$. Now that we have explicitly constructed  $U(t)$ we define
 	\begin{equation}\label{eq defi yrdt}
 		Y_r(t)=U(t)\,X= \begin{pmatrix}
 			I&0 \\
 			0 &C(t)\\
 			0 & S(t)
 		\end{pmatrix} \peso{with} X=\begin{pmatrix}
 			I&0\\
 			0&I\\
 			0&0
 		\end{pmatrix}\,,\end{equation} where the matrix block representation above is with respect to the decomposition 
 	
 	\noindent $\C^d=(\cX\cap \cY)\oplus \cS_1\oplus \cS_2$, as before. 
 	
 	\pausa
 	In this case, we have that
 	\begin{equation}\label{eq defi gamma posta1}
 		\gamma(t)=Y_r(t)^*\, L(t)\, Y_r(t)=X^*\,(\,U(t)^* \, L(t)\,U(t)\, )\, X\peso{for} t\in [0,1]\,.
 	\end{equation}
 	Notice that by taking a derivative,
 	\begin{equation}\label{eq derivada de gamma1}
 		\gamma'(t)=X^*\,U(t)^* \, L'(t)\,U(t)\,X + X^*\,(\,(U'(t))^* \, L(t)\,U(t) + U(t)^* \, L(t)\,U'(t)\,  )\,X\,.
 	\end{equation}
 	The first term in Eq. \eqref{eq derivada de gamma1} can be dealt with using Lidskii's inequality in a simple way.
 	Thus, we are left with the analysis of the second term. Using basic trigonometric identities 
 	we get that 
 	$U(t+h)=U(t)\,U(h)$ for $t,\,h,\,t+h\in[0,1]$; hence $U'(t)=U(t)\,U'(0)$ and similarly $(U'(t))^*=-U'(0)\,U(t)^*$. Therefore, if we let 
 	\begin{equation}\label{eq def A de t}
 		A(t)=U(t)^*\, L(t)\,U(t)
 	\end{equation} then we have that 
 	\begin{equation} \label{eq el segund term de gamma prima1}
 		X^*\,(\,(U'(t))^* \, L(t)\,U(t) + U(t)^* \, L(t)\,U'(t)\,  )\,X=
 		X^*\,(A(t)\,U'(0)-U'(0)\,A(t))\,X\,.
 	\end{equation}
 	By taking a derivative in Eq. \eqref{eq defi udt} we get that 
 	$$
 	U'(0)
 	=\begin{pmatrix}
 		\ 0&0 &0\\
 		\ 0&0&-D_{\Theta'}\\
 		\ 0& \ D_{\Theta'}&0\end{pmatrix}\in i\cdot \matsad\,.
 	$$ where $D_{\Theta'}$ is the diagonal matrix with main diagonal $\Theta'=(\frac{\pi}{2}\,\uno_p\coma \theta_1,\ldots,\theta_{r})$ defined above, and
 		$i\cdot \matsad$ denotes the space of skew-Hermitian matrices.
 	
 	\pausa 
 	Consider the block matrix representation of $A(t)=(A_{ij}(t))_{i,j=1}^3$ with respect to
 	the decomposition in Eq. \eqref{eqdecompC0}. A direct computation, using the previous block matrix representations, shows that 
 	\begin{equation} \label{eq deriv para teo 2sb}
 		X^*\,(A(t)\,U'(0)-U'(0)\,A(t))\,X=2\ \text{Re}\,(
 		\begin{pmatrix}
 			0& A_{13}(t)  \\ 0& A_{23}(t)
 		\end{pmatrix}\,\begin{pmatrix} 0 \\ D_{\Theta'} \end{pmatrix})\,.
 	\end{equation}
 	
 	\begin{rem}\rm
 		With the previous notation, using
 		Weyl's multiplicative inequality (see Theorem \ref{teo ag}) we conclude that 
 		\begin{equation}\label{eq para teo cot dercanpat}
 			s(X^*\,(A(t)\,U'(0)-U'(0)\,A(t))\,X)\prec_w 2 \, \Theta\, s\left(\,\begin{pmatrix}
 				A_{13}(t)  \\ A_{23}(t)
 			\end{pmatrix}\,\right)\,,\end{equation}
 		where we have used that $\Theta=(\Theta',0_s)\in(\R_{\geq 0}^k)\da$ and item 2 in Remark \ref{rem acuerdos}.
 		The previous facts suggest to develop a bound for the singular values of the anti-diagonal block of $A$ according to the representation
 		\begin{equation}\label{eq bloque antid}
 				\left(
 				\begin{array}{cc|c}
 					A_{11}(t)&A_{12}(t) &A_{13}(t)\\
 					A_{21}(t)&A_{22}(t)&A_{23}(t)\\
 					\hline
 					A_{31}(t)&A_{32}(t) &A_{33}(t)
 				\end{array}
 				\right)
 		\end{equation}
 		%
 		(that formally corresponds to the decomposition $\C^d=(\cX\cap\cY\oplus \cS_1)\oplus \cS_2$, where $\cS_1$ and $\cS_2$ are described after Eq. \eqref{eqdecompC0}) in terms of spectral properties of the matrix $A(t)\in\cH(d)$ (defined in Eq. \eqref{eq def A de t}). \EOE
 	\end{rem}
 	
 	\pausa
 	Next, we describe a bound of the singular values of the anti-diagonal 
 	block as in Eq. \eqref{eq bloque antid} in terms of the spectral spread of $A(t)$.

 	\begin{teo}[See \cite{MSZ2}]\label{vale con 2 gral}
 		Let $H =\bm{cc}H_{11}&H_{12} \\H_{12}^*&H_{22} \em \barr {c}\C^k\\ \C^r\earr \in \cH(k+r)$ be arbitrary. Then 
 		\rm
 		\begin{equation}\label{con 2 gral}
 			2\,s(H_{12})\prec_w \sprm(H) \ . 
 		\end{equation} \QED
 	\end{teo}
 	
 	\pausa We can now prove our first main result. We will re-write the statement for the reader's convenience.
 	
 	\pausa {\bf Theorem \ref{teo principal1}.\ }\rm
 	Let $A,\, B\in\cH(d)$ and let $\cX,\,\cY\subset \C^d$ be $k$-dimensional subspaces.
 	Let $U=U(\cX,\cY)$ be a direct rotation of $\cX$ onto $\cY$ and $\Theta=\Theta(\cX,\cY)\da  \in [0,\pi/2]^k$ 
 	the principal angles between $\cX$ and $\cY$ defined in \eqref{theta}. Given 
 	$X\in\cI_\cX(k,d)$, if we let 
 	$$Y_r=Y_r(X,U)\igdef UX\in \cI_\cY(k,d)$$ then we have that 
 	$$
 	s(X^*\,A\,X-Y_r^*\,B\,Y_r )\prec_w s(A-B)+ 	
 	\Theta(\cX,\cY)\da\, \left( \frac{ \,\sprm \,(A)  + \sprm\, (B)}{2}\right)\,,
 	$$
 	where submajorization, sums and products are as in Remark \ref{rem acuerdos}. 
 	
 	\pausa{\it Proof.}  Assume first that $\cX^\perp \cap \cY^\perp=\{0\}$.
 	Let $\gamma(\cdot):[0,1]\rightarrow \cH(k)$ be defined as in Eq. \eqref{eq defi gamma posta1}, where $Y_r(\cdot)$ is defined as in Eq. \eqref{eq defi yrdt}. Notice that by construction $\gamma$ is a smooth curve such that $\gamma(0)=X^*\,A\,X$ and $\gamma(1)=Y_r^*\, B \,Y_r$\,. By Proposition \ref{prop estimando distancias1} we have that 
 	\begin{equation} \label{eq para estimar distancia main1}
 		s(X^*\,A\,X-Y_r^*\,B\,Y_r )\prec_w  \int_0^1 s(\gamma'(t))\ dt\,.
 	\end{equation}
 	
 	\pausa
 	Using Eq. \eqref{eq derivada de gamma1} and Weyl's inequality for singular values (see Theorem \ref{teo ag}), we have that 
 	\begin{equation} \label{eq weyl para gamma deriv1}
 		\barr{rl}
 		s(\gamma'(t)) & \prec_w s(X^*\,U(t)^* \, L'(t)\,U(t)\,X) \\&\\& \quad
 		+ s(X^*\,(\,(U'(t))^* \, L(t)\,U(t) + U(t)^* \, L(t)\,U'(t)\,  )\,X)\,, \earr
 	\end{equation} 
 	where $L(t)=(1-t)\,A+t\,B$, for $t\in [0,1]$ as before. Hence
 	$L'(t)=B-A$ and therefore
 	$$s(X^*\,U(t)^* \, L'(t)\,U(t)\,X)=s(X^* \, U(t)^* \,(B-A)\, U(t)\,X)\prec_w s(A-B)\,,$$ where we have used Remark \ref{remxleqy}. The previous facts together with item 3. in Lemma \ref{lem submaj props1} imply that 
 	\begin{equation} \label{eq prim desig para gamma deriv1}
 		\int_0^1 s(X^*\,U(t)^* \, L'(t)\,U(t)\,X)\ dt\prec_w s(A-B)\,.
 	\end{equation}
 	Moreover, using Eqs. \eqref{eq el segund term de gamma prima1},  \eqref{eq deriv para teo 2sb}, \eqref{eq para teo cot dercanpat}
 	and Theorem \ref{vale con 2 gral} (also see item 5. in Lemma \ref{lem submaj props1}) we get that 
 	\begin{equation} \label{eq desi la otra part de gamma deriv1}
 		s(X^*\,(\,(U'(t))^* \, L(t)\,U(t) + U(t)^* \, L(t)\,U'(t)\,  )\,X)\prec_w \Theta\,\sprm (A(t))\,.
 	\end{equation}
 	Notice that $A(t)$ and $L(t)$ are unitary conjugates, so they have the same spectral spread.
 	Hence, by Proposition \ref{spreadskii aditivo} we get that
 	\begin{equation} \label{eq relac spe spre1}
 		\sprm(A(t))=\sprm(L(t))\prec (1-t)\, \sprm(A) + t\,\sprm(B)\,.
 	\end{equation}
 	By Eqs. \eqref{eq desi la otra part de gamma deriv1} , \eqref{eq relac spe spre1} 
 and item 3. in Lemma \ref{lem submaj props1},
 	\begin{eqnarray*}
 		& &\int_0^1 s(X^*\,(\,(U'(t))^* \, L(t)\,U(t) + U(t)^* \, L(t)\,U'(t)\,  )\,X)\ dt \prec_w \\ & &
 		\Theta\, \left(
 		\sprm(A) \int_0^1  (1-t) \ dt + \sprm(B) \int_0^1  t \ dt \right)
 		= \frac{1}{2}\,\Theta\, (\,\sprm(A)+\sprm(B)\,) \,.
 	\end{eqnarray*}
 	Using this last submajorization relation, together with Eq. \eqref{eq weyl para gamma deriv1} and Eq. \eqref{eq prim desig para gamma deriv1} we get that 
 	\begin{equation}\label{eq desi para gamma der main2}
 		\int_0^1 s(\gamma'(t))\ dt\prec_w s_j(A-B)_{j=1}^k+ 	
 		\Theta\, \left( \frac{ \,\sprm \,(A)  + \sprm\, (B)}{2}\right)\,,
 	\end{equation} where we have used the first part of item 3. in Lemma \ref{lem submaj props1}.
 	The result now follows from Eq. \eqref{eq para estimar distancia main1} and \eqref{eq desi para gamma der main2}. 
 	
 	\pausa In case $\cX^\perp\cap\cY^\perp\neq \{0\}$, we consider the subspace $\cZ=\cX+\cY\subset\C^d$ and 
 	$Z\in \cI_\cZ(m,d)$, where $m=\dim\,\cZ$. Let $A'=Z^*\,A\,Z,\,B'=Z^*\,B\,Z\in \cH(m)$ and let 
 	$X'=Z^*\,X\in \cI(k,m)$. If we let $\cX'=R(X')$ and $\cY'=R(Z^*\,Y_r)$ denote the subspaces that are the ranges of $X'$ and $Z^*\,Y_r$ in $\C^m$ we get that $(\cX')^\perp \cap(\cY')^\perp=\{0\}$. 
 	Moreover, since $(X')^*Y'=(X^*Z)(Z^*Y_r)= X^*\,P_\cZ\,Y_r=X^*Y_r$ (where $P_\cZ\in\cH(d)$ denotes the orthogonal projection onto $\cZ\subset \C^d$) we conclude that $\Theta(\cX',\cY')=\Theta(\cX,\cY)\in\R^k$, by definition of principal angles.
 	By the first part of the proof, we now have that 
 	\begin{equation}\label{eq teo main3}
 		s((X')^*\,A'\,X'-(Y_r')^*\,B'\,Y_r' )\prec_w s_j(A'-B')_{j=1}^k+ 	
 		\Theta\, \left( \frac{ \,\sprm \,(A')  + \sprm\, (B')}{2}\right)\,,
 	\end{equation}
 	where $Y'_r=U'\,X'$, and $U'$ is a direct rotation of $\cX'$ onto $\cY'$.  But notice that 
 	$Z^*UZ$ is a direct rotation from $\cX'$ onto $\cY'$ so 
 	we can take
 	$Y'_r= Z^* \,U\,Z \,Z^*X=Z^*\,Y_r$.
 	Hence, 
 	$$(X')^*\,A'\,X'=X^* \,P_{\cZ}\,A\,P_{\cZ}\,X=X^*\,A\,X \py 
 	$$
 	$$
 	(Y_r')^*\,B'\,Y_r' =Y_r^*\, P_\cZ\, B\, P_\cZ\,Y_r= Y_r^*\,B\,Y_r $$ 
 	Therefore, 
 	\begin{equation} \label{eq teo main4}
 		s((X')^*\,A'\,X'-(Y_r')^*\,B'\,Y_r' )=s(X^*\,A\,X-Y_r^*\,B\,Y_r)\,.
 	\end{equation} On the other hand, using Remark \ref{remxleqy} (or the interlacing inequalities) we now see that
 	$$
 	s(A'-B')=s(Z^*\,(A-B)\,Z)\prec_w s(A-B)\py $$ 
 	$$\sprm(A')\prec_w \sprm(A) \ \ , \ \ \sprm(B')\prec_w \sprm(B)\,.
 	$$ Thus, in this case the result follows from the previous submajorization relations together with Eq. \eqref{eq teo main3} and Eq. \eqref{eq teo main4}.
 	\QED

 	\subsection{Proof of Theorem \ref{teo principal2}}\label{subsec main2}
 	
 	\begin{nota}\label{Not 3.3}
 		Throughout this section we adopt the notation and terminology in Theorem \ref{teo principal2}. Hence,
 		we consider:
 		\ben
 		\item $A\in\cH(d)$;
 		\item  $\cX,\,\cY\subset \C^d$, $k$-dimensional subspaces such that $\cX$ is $A$-invariant;
 		\item a direct rotation $U=U(\cX,\cY)\in \matud$, from $\cX$ onto $\cY$.
 		\item $Y_r=Y_r(X,\cY) \igdef UX\in \cI_\cY(k,d)$. 
 		\item We denote by $\Theta=\Theta(\cX,\cY)\da  \in [0,\pi/2]^k$ 
 		the principal angles between $\cX$ and $\cY$. \EOE
 		\een
 	\end{nota}
 	
 	\pausa
 	Our approach to prove the result is similar to that in the previous section. We first prove Theorem \ref{teo principal2} in the particular case that $k=\dim\cX=\dim\cY\leq d/2$ (see Propositions \ref{pro para teo 3.2} and \ref{para teo 3.2. un poco mas} below). Then, we reduce the general case to the previous particular case. 
 	In what follows we consider the decomposition 
 	$$\C^d=\cX\oplus \cX^\perp\peso{assuming that} \cX\cap \cY =\{0\} \py \cX^\perp\cap\cY^\perp=\{0\}\,.$$ 
 \pausa Notice that in this case we have that 
 		\begin{equation}\label{eq agregada 123}
 			A=\begin{pmatrix} 
 				A_{11} & 0 \\ 0 & A_{22}
 			\end{pmatrix} \peso{i.e.} A_{12}=A_{21}^*=0 \,.
 	\end{equation}
 	Given a direct rotation $U$ from $\cX$ onto $\cY$ we consider the block representation 
 	$$
 	U(t)=\begin{pmatrix} 
 		C(t) & -S(t) \\ S(t) & C(t)
 	\end{pmatrix}\,,
 	$$
 	where $S(\cdot),\,C(\cdot): \ [0\, ,\,1] \rightarrow \H(r)$ are given by 
 	$$
 	S(t)=\text{diag}(\,\sin(t\,\Theta)\,)\py C(t)=\text{diag}(\,\cos(t\,\Theta)\,)\in\cM_{p+r}(\C)^+\,.
 	$$ 
 	Here  $\Theta=\Theta(\cX,\cY)\da =(\frac{\pi}{2}\,\uno_p\coma \theta_1,\ldots,\theta_{r})
 	\in [0,\pi/2]^k$   
 	where, as before,  
 	$$
 	p=\dim\cX\cap \cY^\perp=\dim\cX^\perp\cap \cY  \py 
 	r=\dim \cX\ominus (\cX\cap \cY^\perp)=k-p \ . 
 	$$
 	In this case, we define
 	\begin{equation}\label{eq defi yrdt2}
 		Y_r(t)=U(t)\,X= \begin{pmatrix}
 			C(t)\\
 			S(t)
 		\end{pmatrix}\peso{with} X=\begin{pmatrix}
 			I\\
 			0
 		\end{pmatrix}\,,\end{equation} 
 	where the matrix block representation  is with respect to the decomposition $\C^d=\cX\oplus \cX^\perp$, as before. 
 	In this case, we set
 	\begin{equation}\label{eq defi gamma posta4}
 		\gamma(t)=Y_r(t)^*\, A\, Y_r(t)=X\,(\,U(t)^* \, A\,U(t)\, )\, X=X^*\,A(t)\,X\peso{for} t\in [0,1]\,,
 	\end{equation} where $A(t)=U(t)^*\,A\,U(t)$. By taking derivatives,
 	\begin{equation}\label{eq derivada de gamma2}
 		\gamma'(t)= X^*\,A'(t)\,X=X^*\,(A(t)\,U'(0)-U'(0)\,A(t))\,X\,,
 	\end{equation} where we have used that $U(t+h)=U(t)\,(h)$, for $t,\,h,\,t+h\in [0,1]$, so that 
 	$U'(t)=U(t)\,U'(0)$ and similarly $(U'(t))^*=-U'(0)\,U(t)^*$.
 	In a similar way as in Eq. \eqref{eq deriv para teo 2sb}, 
 	using the previous block matrix representations and setting $A(t):=(A_{ij}(t))_{i,j=1}^2$, 
 	we have that
 	\begin{equation} \label{eq deriv para teo 2sb2}
 		\gamma'(t)= X^*\,(A(t)\,U'(0)-U'(0)\,A(t))\,X=2\ \text{Re}(\,A_{12}(t) \,D_{\Theta})\,.
 	\end{equation}
 	Now, using the fact that $\cX$ is an $A$-invariant subspace(so that $A$ has the block representation in Eq. \eqref{eq agregada 123}), if we consider the block matrix representation $A(t)=U(t)^*\,A\,U(t)=(A_{ij}(t))_{i,j=1}^2$ then, for $t\in [0,1]$
 	\begin{equation}\label{eqqueseyocuanto111}
 		A_{12}(t)=S(t)\, A_{22}\,  C(t)-C(t)\,  A_{11}\,  S(t)\, .
 	\end{equation} 
 	
 		Equation \eqref{eqqueseyocuanto111} 
 		shows that $\cX$ might not be $A(t)$-invariant, for $t\in (0,1]$; nevertheless, based on the fact that $\cX$ is $A$-invariant, we show that it is possible to obtain a convenient upper bound for the growth of the singular values $\gamma'(t)$, for $t\in(0,1]$.
 	Indeed, using Eq. \eqref{eqqueseyocuanto111}
 	we see that 
 	$$
 	2\, \text{Re}(A_{12}(t)\ D_{\Theta})= 2\, \text{Re}(\, (S(t)\, A_{22}\,  C(t)-C(t)\,  A_{11}\,  S(t)\,) \, D_{\Theta})
 	$$
 	Moreover, since $S(t),\, C(t)$ and $D_{\Theta}$ commute with each other, then
 	\begin{equation} \label{eq la pinta de la deriv cas invariant}
 		2\, \text{Re}(A_{12}(t)\ D_{\Theta})= 2\, \text{Re}(\, S(t)(A_{22}\, D_{\Theta} - D_{\Theta }\, A_{11})\, C(t)\,)\,.
 	\end{equation}
 	Using Weyl's multiplicative inequality (item 3 of Theorem \ref{teo ag}) 
 	and  that given $E\in\cM_k(\C)$, then	$s(\text{Re}(E))\prec_w s(E)$, 
 	we get  from Eqs. \eqref{eq deriv para teo 2sb2} and 
 	\eqref{eq la pinta de la deriv cas invariant} that 
 	\begin{equation}\label{eq una acotacot2} 
 		s(\gamma'(t)\,) = 2\, s\left(\,\text{Re}(A_{12}(t)\ D_{\Theta})\,\right)\prec_w 2\,
 		t \,\Theta\,  s( A_{22}\, D_{\Theta} - D_{\Theta}\, A_{11})\,,
 	\end{equation}
 	where we have used that $s(S(t)\,) = \la(S(t)\,) = \sin(t \,\Theta) \leqp t \,\Theta$ 
 	and that $0\leq C(t)\leq I$.
 	
 	\pausa We will also consider the following
 	inequality for the singular values of the so-called generalized commutators. 
 	
 	\begin{teo}[See \cite{MSZ2}]\label{teo gen Kitt} 
 		\rm Let $A_1,\, A_2\in \cH(k)$ be arbitrary and let $D\in\cM_k(\C)$. Then \rm
 		\begin{equation} 
 			s(A_1\,D-D\,A_2)\prec_w s(D) \,\, \sprm\,( A_1 \oplus A_2)\ . 
 		\end{equation}
 	\end{teo}

 	\begin{pro}\label{pro para teo 3.2}
 		With the notation and terminology of Notation \ref{Not 3.3}, assume that $\cX\cap \cY=\{0\}$.
 		Then we have that 
 		\begin{equation}\label{eq teo main2bis}
 			s(X^*\,A\,X-Y_r^*\,A\,Y_r )\prec_w \Theta^2 \, \sprm \, A\,,
 		\end{equation} 
 	\end{pro}
 	\pausa{\it Proof.} We first assume that $\cX^\perp\cap \cY^\perp=\{0\}$. Under this hypothesis,
 	we consider the notation and terminology introduced previously in this section.
 	In particular, we consider the smooth curve $\gamma(t)$ for $t\in [0,1]$ introduced in Eq. \eqref{eq defi gamma posta4}. Since $\gamma(0)=X^*AX$ and $\gamma(1)=Y_r^*\,A\,Y_r$ then, Proposition \ref{prop estimando distancias1} shows that 
 	\begin{equation} \label{eq cota para teo princ 2}
 		s(X^*AX-Y_r^*\,A\,Y_r)\prec_w\int_0^1 s(\gamma'(t))\ dt\,.
 	\end{equation} 
 	By Eq. 
 	\eqref{eq una acotacot2} above, 
 	we have that
 	\begin{equation}\label{eq falta poco1}
 		s(\gamma'(t))\prec_w 2\, t \,\Theta\,  s( A_{22}\, D_{\Theta} - D_{\Theta}\, A_{11})\,,
 	\end{equation}
 	By Theorem \ref{teo gen Kitt} we get 
 	\begin{equation}\label{eq una acotacot2b}
 		s( A_{22}\, D_{\Theta} - D_{\Theta}\, A_{11})\prec_w \Theta\,  \,\sprm(A_{11}\oplus A_{22})=\Theta\,  \,\sprm(A)\,,
 	\end{equation} since $A=A_{11}\oplus A_{22}$ by hypothesis ($\cX$ is $A$-invariant).
 	Therefore, by Eqs.  \eqref{eq falta poco1} and 
 	\eqref{eq una acotacot2b}, 
 	$$
 	s(\gamma'(t))\prec_w 2\, t\, \Theta^2\,  \,\sprm(A)
 	\stackrel{\eqref{eq cota para teo princ 2}}{\implies}
 	s(X^*\,A\,X-Y_r^*\,A\,Y_r )\prec_w \Theta^2 \, \sprm \, A
 	\ , 
 	$$
 	where we have used items 5. and 3. in Lemma \ref{lem submaj props1}. 
 	Finally, in case $\cX^\perp\cap \cY^\perp\neq \{0\}$ we can argue by reduction to the case in which $\cX^\perp\cap \cY^\perp= \{0\}$, in the same way as in the last part of the proof of Theorem \ref{teo principal1}, and prove the inequality in Eq. \eqref{eq teo main2bis} (without the restriction $\cX^\perp \cap\cY^\perp=\{0\}$). The details are left to the reader. 
 	\QED
 	
 	\smallskip
 	
 	\pausa
 	\begin{pro}\label{para teo 3.2. un poco mas}
 		With the notation and terminology of Notation \ref{Not 3.3}, assume that 
 		$\cX,\,\cY\subset \C^d$ be $k$-dimensional subspaces with $k\leq d/2$. 
 		Then we have that 
 		\begin{equation}\label{eq teo main2bisbis}
 			s(X^*\,A\,X-Y_r^*\,A\,Y_r )\prec_w \Theta^2 \, \sprm \, A\,,
 		\end{equation} 
 	\end{pro}
 	
 	\pausa{\it Proof}. Assume that $k=\dim\cX=\dim\cY\leq d/2$ and that $\cX$ and $\cY$ are arbitrary subspaces such that $\cX$ is $A$-invariant. If $\cX\cap \cY=\{0\}$ then we conclude that Eq. \eqref{eq teo main2bisbis} holds in this case, by Proposition \ref{pro para teo 3.2}. In case 
 	$\cX\cap \cY\neq\{0\}$ we consider the following
 	
 	\pausa {\it Claim}: for every $t\in [0,1]$ there exist a $k$-dimensional subspace $\cY(t)\subset \C^d$ and a direct rotation $W(t)\in\cU(d)$ from $\cX$ onto $\cY(t)$ in such a way that $\cX\cap\cY(t)=\{0\}$ for $t\in(0,1]$,
 		$$
 		\lim_{t\rightarrow 0^+}P_{\cY(t)}=P_\cY\in\cH(d) \py \lim_{t\rightarrow 0^+}W(t)=U\,.
 		$$
 		Let us assume for a moment that the claim is true. Let $Y_{r,t}=W(t)\,X$, for $t\in [0,1]$, and notice that $\lim_{t\rightarrow 0^+}Y_{r,t}=U\,X=Y_r$ (as in the statement of the theorem). Since $\cX\cap \cY(t)=\{0\}$ for $t\in (0,1]$, by Proposition \ref{pro para teo 3.2} we have that 
 		$$
 		s(X^*AX-Y_{r,t}^*\,A\,Y_{r,t})\prec_w  \Theta(\cX,\,\cY(t))^2\, \sprm(A) \peso{for} t\in (0,1]\,.
 		$$ By continuity of singular values (and hence of principal angles) we conclude that 
 		\begin{eqnarray*}
 			s(X^*AX-Y_{r}^*AY_{r})&=&\lim_{t\rightarrow 0^+} s(X^*AX-Y_{r,\,t}^*\,A\,Y_{r,\,t})\prec_w  \lim_{t\rightarrow 0^+} \Theta(\cX,\,\cY(t))^2\, \sprm(A)\\ & =&  \Theta(\cX,\,\cY)^2\, \sprm(A)\,. 
 		\end{eqnarray*}
 		{\it Proof of the claim}.  In case $\dim(\cX\cap \cY)=s\geq 1$, we notice that $$\dim \cX^\perp\cap\cY^\perp=\dim\cX^\perp+\dim\cY^\perp-\dim (\cX\cap\cY)^\perp\geq s\,.$$ On the other hand, we have the decompositions
 	$$
 	\cX=(\cX\cap \cY) \oplus (\cX\cap \cY^\perp)\oplus(\cX\cap \cG) \py \cY=(\cX\cap \cY) \oplus (\cX^\perp\cap \cY)\oplus (\cY\cap \cG) \,,
 	$$ where $\cG$ denotes the generic part of the pair of subspaces $\cX$ and $\cY$ (see Eq. \eqref{eq defi part gen}).
 	Let $\Theta'=(\frac{\pi}{2}\,\uno_p\coma \theta_1\,\ldots,\theta_r)$ denote the principal angles between $\cS_1:=(\cX\cap \cY^\perp)\oplus(\cX\cap \cG)$ and $(\cX^\perp\cap \cY)\oplus (\cY\cap \cG)$, where $p=\dim \cX\cap \cY^\perp=\dim\cX^\perp\cap \cY$ and $r=\dim \cX\cap \cG=\dim \cY\cap \cG$; then, $\Theta=\Theta(\cX\coma \cY)= (\Theta ' \coma 0_{s})\in[0,\pi/2]^k$, where $k=p+r+s$.
 	
 	\pausa Let $U\in\cU(d)$ be a direct rotation from $\cX$ onto $\cY$. Then there exists $\cB$ an ONB  for $\C^d$, compatible with the decomposition 
 	\begin{equation} \label{eq decom cd sec4}
 		\C^d=(\cX\cap \cY) \oplus \cS_1 \oplus \cS_2\oplus(\cX^\perp\cap \cY^\perp)\,,
 	\end{equation} where $\cS_2=(\cX^\perp\cap \cY)\oplus (\cX^\perp\cap \cG)$, 
 	such that 
 	\begin{equation}\label{eq defi dir rot u2}
 		U=\begin{pmatrix}
 			I&0 &0 & 0\\
 			0&C&-S & 0\\
 			0&S&C& 0\\
 			0&0 &0 & I
 		\end{pmatrix}\,,\end{equation}  
 	$$\peso{where} C=\text{diag}(\,\cos(\Theta')\,)\py S=\text{diag}(\,\sin(\Theta')\,)\in\cM_{p+r}(\C)^+ \,$$ 
 	are diagonal positive semidefinite matrices. 
 	
 	\pausa Let $\{v_j\}_{j\in\I_s}$ be the vectors in $\cB$ that form an ONB of $\cX\cap \cY$ and let $\{w_j\}_{j\in\I_s}$ be vectors in $\cB$ that form  an ONS in $\cX^\perp\cap\cY^\perp$ (here  we are using that $\dim \cX^\perp\cap\cY^\perp\geq s$). We now set
 	$u_j(t)=\cos(t) \, v_j+\sin(t) \, w_j$, for $j\in\I_s$, and let $\cS(t)$ be the subspace generated by the ONS $\{u_j(t)\}_{j\in\I_s}$, for $t\in [0,1]$. Also, let 
 	$$ \cY(t)=\cS(t)\oplus (\cX^\perp\cap \cY)\oplus (\cY\cap \cG)\peso{for} t\in [0,1]\,.$$
 	By construction, $\dim\cY(t)=k$, and for $0<t\leq \theta_r$ we have that 
 	$$\Theta(\cX,\cY(t))=(\theta_j(t))_{j\in\I_k}=(\Theta ' \coma t\,\uno_{s})\in (0,\pi/2]^k\,.$$
 	In particular, $\cX\cap \cY(t)=\{0\}$ for $0<t<\theta_r$. It is clear that $\lim_{t\rightarrow 0^+}P_{\cY(t)}=P_\cY$.

 	\pausa  Consider the orthogonal decomposition 
 	\begin{equation}\label{eq decom con zetas}
 		\C^d=\cZ_1\oplus\cZ_2\oplus \cT\end{equation} where $\cT=(\cX^\perp\cap \cY^\perp)\ominus \cS(\frac{\pi}{2})$,
 	$$\cZ_1 = (\cX\cap \cY) \oplus(\cX\cap \cY^\perp)\oplus(\cX\cap \cG) \py \cZ_2=(\cX^\perp\cap \cY)\oplus(\cX^\perp\cap \cG)\oplus \cS(\frac{\pi}{2}) \,. $$ 
 	We now construct the block matrix $W(t)$ in terms of the decomposition in Eq. \eqref{eq decom con zetas}, given by
 	\begin{equation} \label{eq defi wdet}
 		W(t)=\begin{pmatrix}
 			C_1(t) & -S_1(t) & 0\\ 
 			S_1(t) & C_1(t) &0 \\
 			0& 0 & I
 		\end{pmatrix}\peso{for} t\in [0,1]\,,
 	\end{equation}
 	where $C_1(t)=\text{diag}(\cos(t)\, \uno_s ,\,\cos(\Theta'))$ and 
 	$S_1(t)=\text{diag}(\sin(t)\, \uno_s ,\, \sin(\Theta'))$. 
 	Then, $W(\cdot)$ is a continuous function such that $W(t)$ is a direct rotation from $\cX$ onto $\cY(t)$, for $t\in [0,1]$, and such that $W(0)=U$ (compare Eqs. \eqref{eq defi dir rot u2} and \eqref{eq defi wdet}).
 	The claim follows from these remarks.
 	\QED
 	
 	\pausa We can now prove our second main result. We will re-write the statement 
 	with all its notation and terminology for the reader's convenience.

 	\pausa
 	{\bf Theorem \ref{teo principal2}.}
 	Let $A\in\cH(d)$ and let $\cX,\,\cY\subset \C^d$ be $k$-dimensional subspaces such that $\cX$ is $A$-invariant. 
 	Let $U=U(\cX,\cY)$ be a direct rotation of $\cX$ onto $\cY$ and $\Theta=\Theta(\cX,\cY)\da  \in [0,\pi/2]^k$ 
 	the principal angles between $\cX$ and $\cY$ defined in \eqref{theta}. 
 	Given $X\in\cI_\cX(k,d)$, if we let $Y_r=Y_r(X,U)=UX\in \cI_\cY(k,d)$ then we have that 
 	\begin{equation}\label{eq teo main2bisbisbis}
 		s(X^*\,A\,X-Y_r^*\,A\,Y_r )\prec_w \Theta^2(\cX,\cY)\da \  \sprm \, (A)\,,
 	\end{equation} 
 	where submajorization and products are as in Remark \ref{rem acuerdos}. 
 	
 	\pausa{\it Proof}. 
 	In case $k\leq d/2$, the result follows from Proposition \ref{para teo 3.2. un poco mas}. On the other hand, if $k> d/2$ we embed the subspaces $\cX$, $\cY$ in $\C^{d'}=\C^d\oplus \C^{(d'-d)}$ for some $k\leq d\,'/2$ and get
 	$\cX'=\cX\oplus 0_{(d'-d)}$ and $\cY'=\cY\oplus 0_{(d'-d)}$. Notice that in this case $\Theta(\cX',\cY')=(\Theta(\cX,\cY),0_{(d'-d)})$.
 	
 	\pausa 
 	Similarly, we can embed $X$, $A$ and $U$  and get 
 	$X'=
 	\begin{pmatrix}
 		X \\ 0
 	\end{pmatrix} \in \cI_{\cX'}(k,d\,') $ , 
 	$$A'=\begin{pmatrix}
 		A & 0 \\ 
 		0 & \la_h(A)\,I 
 	\end{pmatrix}\in \cH(d\,')
 	\py 
 	U'=\begin{pmatrix}
 		U & 0 \\ 
 		0 & I 
 	\end{pmatrix}\in \cU(d\,')
 	\,,$$ where $h = [\frac {d+1}{2}]$ (integer part). 
 	In this case $U'$ is a direct rotation of $\cX'$ onto $\cY'$ such that 
 	$$Y_r'= U'X'=\begin{pmatrix} 
 		UX \\ 0
 	\end{pmatrix}=
 	\begin{pmatrix} 
 		Y_r \\ 0
 	\end{pmatrix}\,.$$
 	Now it is straightforward to check that 
 	$$
 	(X')^*A'\,X'-(Y_r')^*A'\,Y_r'=\begin{pmatrix} 
 		X^*AX-Y_r^*AY_r \\ 0
 	\end{pmatrix}\,.
 	$$Hence, by Proposition \ref{para teo 3.2. un poco mas} we get that 
 	$$s(X^*AX-Y_r^*AY_r )=s((X')^*A'\,X'-(Y_r')^*A'\,Y_r')\prec_w \Theta(\cX',\cY')^2 \, \sprm (A')\,. $$
 	Notice that, by construction, $$\sprm (A')=(\sprm (A), 0_{h'}) \implies \Theta(\cX',\cY')^2 \, \sprm (A')=
 	\Theta(\cX,\cY)^2 \, \sprm (A)\ ,$$ where $h'= [\frac{d'}{2}]-[\frac{d}{2}]\geq 1$.
 	\QED

 		\begin{rem}[Final comments]\label{rem coments finales}\rm
 			Consider the notation of Theorem \ref{teo applica1}. The reader could wonder why is it that our bounds do not coincide with the bounds in Eqs. \eqref{eq conj KA 1} and \eqref{eq conj KA 2}. Now that we have described our techniques in detail we can give our opinion on this issue. We believe that 
 			the fact that our bounds depend on the $\Theta(\cX,\cY)$, which from a geometric point of view corresponds to the arc length between subspaces, is a consequence of our choice of curves that lead to the construction of $\gamma(t)$ as in Eqs. \eqref{eq defi gamma posta1} and  \eqref{eq defi gamma posta4}, based on direct rotations. On the other hand $\sin(\Theta(\cX,\cY))$, which from a geometric point of view corresponds to the chordal length between subspaces, seems to be associated with shorter curves $\tilde \gamma(t)$. Still, the importance of our choice $\gamma(t)$ is that its derivative $\gamma'(t)$ leads to singular values inequalities that allow to reduce the problem to the infinitesimal level (these inequalities where obtained in \cite{MSZ2}, motivated by the problems in the present paper). For example, notice that the shortest curve $\mu(t)=(1-t)X^*AX-t Y^*AY$, $t\in [0,1]$, joining $X^*AX$ and $Y^*AY$ does not provide such reduction, since $\mu'(t)=Y^*AY-X^*AX$, for $t\in[0,1]$. Nevertheless, the geometric technique considered in here is rather flexible in several ways, and we plan to keep working in these problems in the future. \EOE
 	\end{rem}
 	
 	\section{Appendix}\label{appendix}

 	Here we collect several results about majorization and the spectral spread of self-adjoint matrices, used throughout our work.

 	\subsection{Majorization theory in matrix analysis}\label{subsec append mayo}
 	
 	There are many fundamental results in matrix theory that are stated in terms of submajorization relations. Below, we mention only those results related to the content of this work; for a detailed exposition on majorization theory see \cite[Chapters II and III]{bhatia}, 
 	\cite[Chapter 3]{HJ} and \cite[Chapter 9]{MaOlAr}. 
 	
 	\pausa
 	The first result deals with submajorization relations between singular values of arbitrary matrices in $\mat$ (see \cite{bhatia} p.\,35, p.\,74 and p.\,94).
 	
 	\begin{teo}\label{teo ag}\rm Let $C,\,D\in \cM_d(\C)$. Then,
 		\ben
 		\item Weyl's additive inequality for singular values : $s(C+D)\prec_w s(C)+s(D)$;
 		\item $s(\text{Re}(C))\prec_w s(C)$ ;
 		\item Weyl's multiplicative inequality for singular values: $s(CD)\prec_w s(C)\, s(D)$;
 		\een\QED
 	\end{teo}

 	\begin{teo}\label{teo ah}(\cite{bhatia}, Thm.III.4.1) \rm Let $C,\, D\in \matsad$. Then,
 		\ben
 		\item Lidskii's additive inequality: $\lambda(C)-\lambda(D)\prec \lambda(C-D)$;
 		\item Lidskii's additive inequality for singular values: $|\la(C)-\la(D)|\prec_w s(C-D)$;
 		\QED\een 
 	\end{teo}
 	%
 	
 	\pausa
 	In the next result we describe several elementary but useful properties of majorization 
 	and submajorization between real vectors.
 	
 	\begin{lem}\label{lem submaj props1}\rm 
 		Let $x\coma y\coma  z\coma w \in \R^k$. Then,
 		\ben
 		\item $x\da + y\ua\prec x+y\prec x\da+y\da$;
 		\item If $x\prec_w y$ and $y,\, z\in(\R^k)\da$ then $x+z\prec_w y+z$;
 		\item[3.] If $z\coma w\in (\R^k)\da$,  $x\prec z$ and $y\prec w$ then $ x+ y\prec z+w$. 
 		Moreover, if $x(\cdot),\,z(\cdot):[0,1]\rightarrow \R^k$ are continuous functions such that 
 		$x(t)\prec_w z(t)=(z(t))\da$ for $t\in [0,1]$, then 
 		$$
 		\int_0^1x(t)\ dt\prec_w \int_0^1z(t)\ dt\,.
 		$$
 		\een
 		If we assume further that $x\coma y\coma  z\in \R_{\geq 0}^k$ then,
 		\ben
 		\item[4.] $x\da\, y\ua\prec_w x\, y\prec_w x\da\, y\da$;
 		\item[5.] If $x\prec_w y$ and $y,\, z\in (\R_{\geq 0}^k)\da$ then $x\, z\prec_w y\, z$.
 		\QED
 		\een
 	\end{lem}
 	\begin{proof}
 		A proof of all these facts can be found in \cite{bhatia} with the exception of the second assertion in item 3, that we now prove. Indeed, for $n\in\N$ consider the regular partition $\{t_0=0<t_1<\ldots<t_n=1\}$ of $[0,1]$ so that $t_j=\frac{j}{n}$ for $ j\in\{0\}\cup\I_n$ and $\Delta_j=t_j-t_{j-1}=\frac{1}{n}=\Delta_n$, for $j\in\I_n$. By the first part of item 3. we have that, for $n\in \N$, 
 			$$
 			\sum_{j=1}^n x(t_j)\,\Delta_n\prec_w \sum_{j=1}^n z(t_j)\,\Delta_n\,.
 			$$
 			Since $x(\cdot),\,z(\cdot)$ are continuous functions, we have that 
 			$$
 			\int_0^1x(t)\ dt=\lim_n \sum_{j=1}^n x(t_j)\,\Delta_n\prec_w \lim_n \sum_{j=1}^n z(t_j)\,\Delta_n= \int_0^1z(t)\ dt\,,
 			$$ where we have used the following fact: 
 			if 
 			$(u_n)_{n\in\N}$ and $(v_n)_{n\in\N}\in\R^d$ are sequences that converge to $u$ and $v\in\R^d$ respectively, and such that $u_n\prec_w v_n$ for $n\in\N$, then $u\prec_w v$.
 	\end{proof}

 	\begin{rem}\label{remxleqy}\rm
 		Let $x,y\in \R^k$. If $x\leqp y$ then, 
 		$$
 		x^{\downarrow}\leqp y^{\downarrow}\  \text{ and } \ x\prec_w y \, .
 		$$
 		In particular, if $A,\,P\in\mat$ are such that $P$ is a projection then $s(PAP)\leqp s(A)$ so that $s(PAP)\prec_w S(A).$
 		\EOE\end{rem} 
 	
 	\begin{pro}\label{hat trick como en el futbol}\rm
 		Let $1\leq k<d$ and let $E\in \cM_{k,(d-k)}(\C)$. Then 
 		\begin{equation}
 			\hat{E}=\begin{pmatrix} 0&E\\ E^*&0
 			\end{pmatrix}\in \H(d) \py \la(\hat E)= (s(E),-s(E^*))\da\in(\R_{\geq 0}^d)\da\ .
 		\end{equation}
 		\QED
 	\end{pro}

 	\subsection{Spectral spread of self-adjoint matrices}\label{spread sec}
 	
 	In this section we describe a Weyl's type inequality for the spectral spread.
 	Recall that given $A\in \cH(d)$ then the spectral spread of $A$ is 
 	$$
 	\sprm (A) =  \big(\spr_j(A)\, \big)_{j\in \I_h}
 	=\big(\la_j(A)-\la\ua_j(A)\, \big)_{j\in \I_h} \in (\R_+^h)\da \ .
 	$$ where $h = [\frac d 2]$ (integer part).

 	\begin{pro}\label{spreadskii aditivo}
 		Let $A,B\in \H(d)$. Then 
 		$$
 		\sprm(A+B)\prec\sprm(A)+\sprm(B) \,.$$
 	\end{pro}
 	\proof
 	By Lidskii's additive inequality 
 	$$\la(A+B)\prec \la(A)+\la(B) \implies  
 	-\la(A+B)\prec -\la(A)-\la(B) \,.$$
 	By item 3 of Lemma \ref {lem submaj props1}, 
 	since $- \la\ua(A)-\la\ua(B)\in (\R^d)\da$, then 
 	$$
 	\barr{rl}
 	\la(A+B)-\la\ua(A+B)& \prec \la(A)+\la(B) - \la\ua(A)-\la\ua(B) \\&\\&
 	=(\la(A)- \la\ua(A))+(\la(B) -\la\ua(B))\in(\R^d)\da\,. \earr
 	$$ 
 	Let $h = [\frac d 2]$ (integer part). If $r\in\I_h$,
 	\begin{equation}
 		\barr{rcl}
 		\suml_{i=1}^r \sprm_i(A+B)&=&\suml_{i=1}^r \la_i(A+B)-\la_i\ua(A+B)
 		\\ &&\\
 		&\leq &
 		\suml_{i=1}^r (\la_i(A)- \la_i\ua(A))+(\la_i(B) -\la_i\ua(B))\\&&\\
 		&=&
 		\suml_{i=1}^r \sprm_i(A) +\sprm_i(B)\,. 
 		\earr
 	\end{equation}
 	\QED


{\scriptsize
 		\begin{thebibliography}{99}
 			
 			\bibitem{AKPRitz} M.E. Argentati, A.V. Knyazev, C.C. Paige, I. Panayotov, 
 			Bounds on changes in Ritz values for a perturbed invariant subspace of a Hermitian matrix. SIAM J. Matrix Anal. Appl. 30 (2008), no. 2, 548-559.
 			
 			
 			
 			
 			\bibitem {bhatia} R. Bhatia, Matrix analysis, 169, Springer-Verlag, New York, 1997. 
 			
 			\bibitem {BosDr} N. Bosner, Z. Drma${\rm\check{c}}$, Subspace gap residuals for Rayleigh-Ritz approximations. SIAM J. Matrix Anal. Appl. 31 (2009), no. 1, 54--67. 
 			
 			
 			\bibitem{DavKah} C. Davis, W.M. Kahan, The rotation of eigenvectors by a perturbation. III. SIAM J. Numer. Anal. 7 1970 1-46. 
 			
 			\bibitem{GoLoa} G.H. Golub, C.F. Van Loan, Matrix computations. Fourth edition. Johns Hopkins Studies in the Mathematical Sciences. Johns Hopkins University Press, 2013. 
 			
 			\bibitem{Hal} P.R. Halmos, Two subspaces. Trans. Amer. Math. Soc. 144 1969 381-389. 
 			
 			\bibitem{HJ} R.A. Horn, C.R. Johnson, Topics in Matrix analysis, Cambridge University Press, Cambridge, 1991. 
 			
 			
 			
 			\bibitem{AKProxy} A.V. Knyazev, M.E. Argentati, 
 			On proximity of Rayleigh quotients for different vectors and Ritz values generated by different 
 			trial subspaces. Linear Algebra Appl. 415 (2006), no. 1, 82-95.
 			
 			\bibitem{AKMaj} A.V. Knyazev, M.E. Argentati, 
 			Majorization for changes in angles between subspaces, Ritz values, and graph Laplacian spectra. 
 			SIAM J. Matrix Anal. Appl. 29 (2006/07), no. 1, 15-32.
 			
 			\bibitem{AKFEM} A.V. Knyazev, M.E. Argentati, 
 			Rayleigh-Ritz majorization error bounds with applications to FEM. SIAM J. Matrix Anal. Appl. 31 (2009), no. 3, 1521-1537. 
 			
 			
 			
 			
 			\bibitem {LiLi} C.-K. Li, R.-C. Li,  A note on eigenvalues of perturbed Hermitian matrices. Linear Algebra Appl. 395 (2005), 183-190.
 			
 			\bibitem {MaOlAr} A.W. Marshall, I. Olkin, B.C. Arnold, Inequalities: theory of majorization and its applications. Second edition. Springer Series in Statistics. Springer, New York, 2011. 
 			
 			\bibitem {MSZ} P.G. Massey, D. Stojanoff, S. Z\'arate, Majorization bounds for Ritz values of self-adjoint matrices. SIAM J. Matrix Anal. Appl. 41 (2020), no. 2, 554-572.
 			
 			\bibitem {MSZ2} P. Massey, D. Stojanoff, S. Z\'arate, The spectral spread of Hermitian matrices, Linear Algebra Appl. 616 (2021), 19-44.  
 			
 			\bibitem {Mathias} R. Mathias,  Quadratic residual bounds for the Hermitian eigenvalue problem. 
 			SIAM J. Matrix Anal. Appl. 19 (1998), no. 2, 541-550.
 			
 			
 			\bibitem {Ovt}
 			E. Ovtchinnikov, 
 			Cluster robust error estimates for the Rayleigh-Ritz approximation. II. Estimates for eigenvalues. Linear Algebra Appl. 415 (2006), no. 1, 188-209. 
 			
 			\bibitem {Parlett} 
 			B.N. Parlett, 
 			The symmetric eigenvalue problem. Corrected reprint of the 1980 original, 
 			Society for Industrial and Applied Mathematics (SIAM), Philadelphia, PA, 1998.
 			
 			\bibitem {Saad} 
 			Y. Saad, Numerical Methods for Large Eigenvalue Problems, Revised Edition, Society for Industrial and Applied Mathematics (SIAM), Philadelphia, PA, 2011.
 			
 			\bibitem {StewSun} G.W. Stewart, J. Sun, Matrix perturbation theory, Computer Science and Scientific
 			Computing, Academic Press Inc., Boston, MA, 1990. 
 			
 			
 			\bibitem {TeLuLi} Z. Teng, L. Lu, R.-C. Li, 
 			Cluster-robust accuracy bounds for Ritz subspaces. Linear Algebra Appl. 480 (2015), 11-26. 
 			
 			\bibitem {ZAK} P. Zhu, M.E. Argentati, A.V. Knyazev, Bounds for the Rayleigh quotient and the spectrum of self-adjoint operators. 
 			SIAM J. Matrix Anal. Appl. 34 (2013), no. 1, 244-256.
 			
 			
 			\bibitem{ZK} P. Zhu, A.K. Knyazev, Rayleigh-Ritz majorization error bounds of mixed type. SIAM J. Matrix Anal. Appl. 38 (2017), no. 1, 30-49.
 			
 			
 			
 	\end{thebibliography}}
 \end{document}